\documentclass[11pt]{amsart}

\usepackage{color,graphicx,enumerate,wrapfig,amssymb,setspace}

\usepackage{subfigure}

\usepackage{eucal}

\usepackage{hyperref}

\usepackage{enumitem}

\hypersetup{
    bookmarks=true,         
    unicode=true,          
    pdftoolbar=true,        
    pdfmenubar=true,        
    pdffitwindow=false,     
    pdfstartview={FitH},    
    pdftitle={My title},    
    pdfauthor={Author},     
    pdfsubject={Subject},   
    pdfcreator={Creator},   
    pdfproducer={Producer}, 
    pdfkeywords={keyword1} {key2} {key3}, 
    pdfnewwindow=true,      
    colorlinks=true,       
    linkcolor=blue,          
    citecolor=blue,        
    filecolor=magenta,      
    urlcolor=cyan           
}

\renewenvironment{proof}[1][\proofname ]{{\noindent \bfseries #1. }}{\qed \bigskip }

\newcommand{\R}{{\mathbb R}}

\newcommand{\Z}{{\mathbb Z}}
\newcommand{\N}{{\mathbb N}}
\newcommand{\Q}{{\mathbb Q}}
\newcommand{\T}{{\mathbb T}}

\newcommand{\Ss}{{\mathbb S}}
\newcommand{\OO}{{\mathrm O}}

\newcommand{\e}{\varepsilon}

\newcommand{\Div}{\mathrm{div}}

\newcommand{\yy}{{ \widetilde y}}
\newcommand{\zz}{{ \widetilde z}}

\newtheorem{theorem}{Theorem}[section]

\newtheorem{claim}[theorem]{Claim}

\newtheorem{cor}[theorem]{Corollary}

\newtheorem{example}[theorem]{Example}

\newtheorem{lem}[theorem]{Lemma}

\newtheorem{remark}[theorem]{Remark}

\newtheoremstyle{named}{}{}{\itshape}{}{\bfseries}{.}{.5em}{\thmnote{#3 }#1}
\theoremstyle{named}

\numberwithin{equation}{section}

\title[Regularity of the homogenized boundary data]{Regularity of boundary data in periodic homogenization of elliptic systems in layered media}
\author{Hayk Aleksanyan}
\address{School of Mathematics, The University of Edinburgh, JCMB The King's Buildings, Peter Guthrie Tait Road, Edinburgh EH9 3FD}
\curraddr{Department of Mathematics, KTH Royal Institute of Technology, SE-100 44  Stockholm,
Sweden}

\email{hayk.aleksanyan@gmail.com}

\begin{document}

\begin{abstract}
In this note we study periodic homogenization of Dirichlet problem for divergence type elliptic systems
when both the coefficients and the boundary data are oscillating.
One of the key difficulties here is the determination of the fixed boundary data corresponding to the limiting (homogenized) problem.
This issue has been addressed in recent papers by D. G\'{e}rard-Varet and N. Masmoudi \cite{GM1},
and by C. Prange \cite{Prange}, however, not much is known about the regularity of this fixed data.
The main objective of this note is to initiate a study of this problem,
and to prove several regularity results in this connection.


    \end{abstract}


\maketitle

\section{Introduction}\label{sec-Intro}
For a bounded domain $D\subset \R^d$ ($d\geq 2$) consider the following problem
\begin{equation}\label{problem-osc-oper}
 -\nabla \cdot \left(A \left( \frac{ \cdot }{\e} \right) \nabla u \right)(x) =0 , \qquad x\in D,
\end{equation}
with oscillating Dirichlet data
\begin{equation}\label{prob-Dir-data-osc}
 u(x)=g \left(x, \frac{x}{\e} \right), \qquad x \in \partial D.
\end{equation}

\noindent Here $\e>0$ is a small parameter,
$A(x)=(A_{ij}^{ \alpha \beta }(x))$ is $\R^{N^2 \times d^2}$-valued
function defined on $\R^d$, where $1\leq \alpha, \beta \leq d$, $1\leq i,j\leq N$, and
the boundary data $g(x,y)$ is $\R^N$-valued function
defined on $\partial D \times \R^d$.
The action of the operator in (\ref{problem-osc-oper}) on a vector-function $u=(u_1,...,u_N)$ is defined as
$$
-(\mathcal{L}_\e u)_i(x):=\left[ \nabla \cdot \left(A \left(\frac{\cdot}{\e} \right) \nabla u \right) \right]_i (x) =
 \frac{\partial}{\partial x_{\alpha}} \left[ A^{\alpha \beta }_{ij } \left( \frac{\cdot }{\e} \right)
\frac{\partial u_j}{\partial x_{\beta}}   \right] (x),
$$
where $1\leq i \leq N$. Here and throughout the text,
if not stated otherwise, we use the summation convention for repeated indices. 

\vspace{0.1cm}

\noindent \textbf{Assumptions.} Here we collect all assumptions
which will be used when studying problem (\ref{problem-osc-oper})-(\ref{prob-Dir-data-osc}).

\begin{itemize}
\item[(A1)]  (Periodicity) The coefficient tensor $A$
 and the boundary data $g$ in its second (oscillating) variable are $\Z^d$-periodic, that is $\forall y\in \R^d, \ \forall h\in \Z^d$
 and $\forall x \in \partial D$ one has
$$
A(y+h)=A(y), \  g(x,y+h)=g(x,y).
$$

\item[(A2)] (Ellipticity) Coefficients are uniformly elliptic and bounded, that is there exist constants $\Lambda, \lambda>0$ such that
$$
  \lambda \xi_{\alpha}^i \xi_{\alpha}^i \leq A_{ij}^{\alpha \beta}(x) \xi_{\alpha}^i \xi_{\beta}^j \leq \Lambda \xi_{\alpha}^i
\xi_{\alpha}^i, \qquad \forall x\in \R^d, \ \forall \xi \in \R^{d\times N}.
$$

\item[(A3)] (Smoothness) We suppose that the boundary data $g$ in both variables,
 all elements of $A$, and the boundary of $D$ are infinitely smooth.

\item[(A4)] (Geometry of the domain) $D$ is a strictly convex domain, i.e. the all principal curvatures of $\partial D$ are bounded away from zero.

\item[(A5)] (Layered medium structure) We assume that the coefficient tensor $A$ is independent
of some fixed rational direction, i.e. there exists a non-zero vector $\nu_0 \in \Z^d$
such that $(\nu_0 \cdot \nabla) A (y) =0$ for all $y \in \mathbb{T}^d$.
\end{itemize}

\vspace{0.1cm}

The last hypothesis (A5) models 
media with layered structure, for instance,
(A5) includes the class of first order \emph{laminates}.
Although homogenization results concerning laminates have been studied in theory,
and have independent interest (see e.g. \cite{Neuss}), here the assumption (A5)
is technical and is due to our proof. 

For each $\e>0$ let $u_\e$ be the solution to problem (\ref{problem-osc-oper})-(\ref{prob-Dir-data-osc}).
Also, for the family of operators $\{ \mathcal{L}_\e \}_{\e>0}$
let $\mathcal{L}_0$ be the homogenized (effective) operator in a
usual sense of the theory of homogenization (see e.g. \cite{BLP}).
The following homogenization result for $u_\e$ is due to D. G\'{e}rard-Varet, and N. Masmoudi.

\begin{theorem}\label{Thm-GM-main}{\normalfont(see \cite[Theorem 1.1]{GM1})}
Under assumptions (A1)-(A4) there exists a fixed boundary data\footnote{This theorem is formulated in \cite{GM1}
with $g^* \in L^p(\partial D)$ for all finite $p$. However \cite{GM1} contains a proof of the stronger statement $g^* \in L^\infty(\partial D)$,
which we use in the current formulation (in \cite{GM1} see Proposition 2.4, and the discussion at the end of page 159).  }
 $g^* \in L^\infty(\partial D)$ such that if
$u_0$ solves 
$$
\mathcal{L}_0 u_0 (x) = 0 , \ x \in D \qquad{ and } \qquad u_0(x)= g^*(x), \ x \in \partial D,
$$
then
$$
|| u_\e - u_0 ||_{L^2(D)} \leq C_\alpha \e^\alpha, \qquad \forall \alpha \in \left(0, \frac{d-1}{3d+5} \right).
$$
\end{theorem}

A result related to Theorem \ref{Thm-GM-main} was proved
in our recent work \cite{ASS2} in collaboration with H. Shahgholian, and P. Sj\"{o}lin, by an approach different than that of \cite{GM1}.
Define projections $P^k_\gamma(x) = x_\gamma (0,...,1,0,...) \in \R^N$
with 1 in the $k$-th position, where $1\leq \gamma \leq d$ and $1\leq k \leq N$. Also, let
$\mathcal{L}^*_\e$ be the adjoint operator to $\mathcal{L}_\e$, that is the coefficients of $\mathcal{L}^*_\e$
are set as $(A^*)^{\alpha \beta}_{ij} = A^{\beta \alpha}_{ji}$. We then have the following result.

\begin{theorem}\label{Thm-our2}{\normalfont(see \cite[Theorem 1.7]{ASS2})}
In the same setting as in Theorem \ref{Thm-GM-main}, assume in addition that $d\geq 3$
and $\mathcal{L}_\e^* (P_\gamma^k)=0$
in $D$ for all $1\leq k \leq N$, $1\leq \gamma \leq d$, and any $\e>0$. Then
there exists a function $g^*$ infinitely smooth on $\partial D$, so that
if $u_\e$ is the solution to (\ref{problem-osc-oper})-(\ref{prob-Dir-data-osc})
and $u_0$ of that with homogenized operator $\mathcal{L}_0$ and boundary data $g^*$ then
$$
|| u_\e - u_0||_{L^p(D)} \leq C_p [\e (\ln(1/\e))^2]^{1/p},
$$
for any $1\leq p<\infty$. Moreover, $g^*$ may be represented explicitly in
terms of the vector field of normals of $\partial D$, boundary data $g$, the coefficient tensor $A$ and coefficients
of the operator $\mathcal{L}_0$.
\end{theorem}

Using the periodicity condition on the coefficients $A$ one may simplify the condition
of Theorem \ref{Thm-our2} on $P_\gamma^k$-s. Namely, denote $v_{k,i}^\gamma(x): = (A_{ki}^{\gamma 1},..., A_{ki}^{\gamma d})(x)$,
for $x\in \R^d$, $1\leq k,i\leq N$, $1\leq \gamma \leq d$, then it is easy to see that the condition
$\mathcal{L}_\e^* (P_\gamma^k) \equiv 0$ is equivalent to
\begin{equation}\label{div-free-condition}
\Div (v_{k,i}^\gamma)(x) = 0, \ x\in \R^d, \ 1 \leq k,i \leq N, \ 1\leq \gamma \leq d.
\end{equation}
In the case of $N=1$ (scalar equations) the last condition means that the rows of the matrix $A$
considered as vector fields in $\R^d$ must be divergence free. The result concerning regularity of $g^*$
contained in Theorem \ref{Thm-our2},
although restrictive in terms of the structure of the operator $\mathcal{L}_\e$,
shows that in some cases one may have smooth boundary data for the homogenized problem.
Looking ahead let us remark here, that among other things we will recover this result for $g^*$ (see subsection \ref{sec-concluding})
by a different method which will show the smoothness of $g^*$ under conditions of Theorem \ref{Thm-our2}
in dimension two as well.

Departing from here, we aim at understanding the regularity of the fixed boundary data $g^*$
defined by Theorem \ref{Thm-GM-main}. Let us first recall some known facts from \cite{GM1} concerning $g^*$.
For a unit vector $n\in \Ss^{d-1}$ let
$P_{n^\bot} $ be the operator of orthogonal projection on the hyperplane orthogonal
to $n$. Fix $l>0$ so that $(d-1)l>1$ and for $\kappa>0$ set
\begin{equation}\label{Diop-class}
\mathcal{A}_\kappa=\big\{ n\in \Ss^{d-1}: | P_{n^\bot}(\xi) | \geq \kappa |\xi|^{-l} \text{ for all } \xi \in \Z^d \setminus \{0\}   \big\}.
\end{equation}
A vector $n\in \mathbb{S}^{d-1}$ is called \emph{Diophantine}, if $n\in \mathcal{A}_\kappa$ for some $\kappa>0$.
For $x\in \partial D$ let $n(x)$ be the unit inward normal at $x$,
and define $\Gamma_\kappa = \{ x\in \partial D: \ n(x) \in \mathcal{A}_\kappa \}$.
One can see from the analysis of \cite{GM1} that for any $\kappa>0$
the restriction of $g^*$ on $\Gamma_\kappa$ is Lipschitz continuous
with the Lipschitz constant bounded by $C \kappa^{-2}$, where the constant $C=C(A,D, g,d)$.
It is shown in \cite{GM1} that $\sigma (\Ss^{d-1} \setminus \mathcal{A}_\kappa)\leq C \kappa^{d-1}$,
where $\sigma$ denotes the Lebesgue measure
on the unit sphere of $\R^d$. Also, it is not hard to see that the complement $\mathcal{A}_\kappa^c = \Ss^{d-1}\setminus \mathcal{A}_\kappa$,
while a set of small measure, is everywhere dense and is an open subset of the unit sphere.
Next, due to strict convexity of $D$ and smoothness of $\partial D$, we have that 
the Gauss map of $\partial D$, namely $\partial D \ni x \longmapsto n(x)\in \mathbb{S}^{d-1}$ is a diffeomorphism,
which implies that the sets $\Gamma_\kappa$ have similar properties as $\mathcal{A}_\kappa$,
in particular, the surface measure of $\Gamma_\kappa$ decays as $\kappa\to 0$, and the complement of each $\Gamma_\kappa$
is open and dense in $\partial D$. We see that as $\kappa \to 0$, the sets $\Gamma_\kappa$
cover the entire boundary of $D$ up to measure zero, and hence $g^*$ is defined almost everywhere on $\partial D$.
However, since the upper bound for Lipschitz constant of $g^*$ on $\Gamma_\kappa$,
which is $C\kappa^{-2}$, blows up as $\kappa \to 0$, we
\emph{cannot} conclude that there exists an extension of $g^*$ to $\partial D$ which will be continuous at least
at a single point. As we will see here, the behaviour of $g^*$ is more regular for
layered structures. 

For a given domain $D$ with smooth boundary, and
$\tau  >0$ set
$$
\partial D_\tau = \{ x\in \partial D: \hspace{0.03cm} n(x) \notin \R \Q^d \text{ and }  |n(x) \cdot \nu_0 | >\tau \},
$$
where $\nu_0$ is fixed from assumption (A5). We have the following result.

\begin{theorem}\label{thm-Main}{\normalfont(The Regularity Theorem)}
Let assumptions (A1)-(A5) be in force, and let $g^*$ be defined by Theorem \ref{Thm-GM-main}. Then, for 
any $\tau>0$ there exists a constant $C_\tau= C(A, D, g, d, \tau)$ such that
$$
|g^\ast(x) - g^\ast(y)| \leq C_\tau |x-y|, \qquad \forall x,y \in \partial D_\tau.
$$ 
\end{theorem}

\begin{cor}\label{cor-to-main}
$g^*$ has a unique continuous extension to $\{x\in \partial D: \ n(x) \cdot \nu_0 \neq 0\}$.
\end{cor}
\begin{proof}
Note that by Theorem \ref{Thm-GM-main} $g^*$ is defined almost everywhere on $\partial D$ 
and we need to extend $g^*$ on a measure zero set of $\partial D$.
By Theorem \ref{thm-Main} for any $\tau>0$ the function $g^*$ is uniformly continuous on $\partial D_\tau$,
and hence admits a unique continuous extension to $\{x\in \partial D: \ |n(x)\cdot \nu_0 |>\tau  \}$.
The proof follows by taking $\tau \to 0$.
\end{proof}

The next two examples are meant to point out some scenarios when Theorem \ref{thm-Main} can be used more effectively.
\begin{example}\label{ex-1}
Under (A1)-(A4) assume in addition that the coefficient tensor $A$ is independent
of the first $k$ coordinates for some $1\leq k \leq d$,
or equivalently that (A5) is satisfied for vectors $\{e_i \}_{i=1}^k$ where $e_i\in \R^d$ is the $i$-th vector
of the standard basis of $\R^d$.
Then, for each $1\leq i \leq k$ taking $e_i$ as the vector in 
assumption (A5), and applying Theorem \ref{thm-Main} $k$-times, we get
that for any $\tau>0$ there exists a constant $C_\tau= C(A, D, g, d, \tau)$ such that
$$
|g^\ast(x) - g^\ast(y)| \leq C_\tau |x-y|, \qquad \forall x,y \in \partial D_\tau^{(k)},
$$
where 
$$
\partial D_\tau^{(k)} = \{ x\in \partial D: \hspace{0.03cm} n(x) \notin \R \Q^d {\normalfont \text{ and } } \max\limits_{1\leq i \leq k}|n(x) \cdot e_i | >\tau \}.
$$
Likewise, Corollary \ref{cor-to-main} implies that $g^*$ has a unique
continuous extension to $\{x\in \partial D: \hspace{0.03cm} \max_{1\leq i \leq k}|n(x) \cdot e_i | \neq 0 \}$.
This shows that for $1\leq k \leq d-1$ the set of discontinuity of $g^*$ can have Hausdorff dimension at most $d-k-1$,
while in the case of $k=d$, i.e. when the coefficients are constant, one gets that $g^*$ extends continuously on the entire boundary of $D$.
The latter statement matches (in a weaker form) with already known result from \cite{ASS2} where it is proved that for constant
coefficient operators, the homogenized boundary data is the average of $g$ in its periodic
variable, and hence is smooth in particular (see also subsection \ref{sec-concluding}).
\end{example}

\begin{example}
Let the domain $D$ be the unit ball of $\R^d$, and suppose the coefficient tensor $A(x)$ is independent of all variables except possibly
variable $x_i$ for some $1\leq i \leq d$ (i.e. $A$ models a first order laminate). Let also the assumptions (A1)-(A3) be in force.
Clearly $\partial D = \mathbb{S}^{d-1}$ and thus $g^*$ is a function on the unit sphere.
Then, from Example \ref{ex-1} we get that $g^*$ has a unique continuous extension
to the unit sphere, except possibly two poles $(0,...,0, \pm 1 ,0,...,0)\in \mathbb{S}^{d-1}$ where the non-zero element is in the $i$-th coordinate.
\end{example}

The strategy of the proof of Theorem \ref{thm-Main} will be discussed in the next section.
In general, without any structural assumptions on the operator, we do not
know whether $g^*$ has an extension to $\partial D$ which is
continuous at least at a single point on the boundary.
Also, it will be very interesting to see if the regularity of $g^*$ can have some impact on the speed of convergence in the actual homogenization
problem (\ref{problem-osc-oper})-(\ref{prob-Dir-data-osc}). A positive sign in this direction
is Theorem \ref{Thm-our2}, although there the smoothness of $g^*$ is a
corollary, rather than a starting point. 

\vspace{0.2cm}

\noindent \textbf{Notation.}
We fix some notation and conventions that will be used in the sequel.
An integer $d$ always stands for the dimension of $\R^d$, and throughout the paper we have $d\geq 2$.
By $N \in \mathbb{N}$ we denote the number of equations in (\ref{problem-osc-oper}).

$\mathbb{S}^{d-1}$ is the unit sphere, and $\mathbb{T}^d$ is the unit torus of $\R^d$.
By $\R \Q^d$ we denote the set of all vectors from $\R^d$ that are scalar multiples of vectors
with all entries being rational numbers. We call elements of $\R \Q^d$ \emph{rational} vectors (directions, if they have length one),
and the complement of $\R \Q^d$ is referred to as \emph{irrational} vectors (correspondingly directions).

In the sequel notation $d \sigma $ in integrals stands for standard surface measure.

For a vector $n\in \Ss^{d-1}$ we set
$\Omega_n =  \{x\in \R^d: \  x \cdot n >0 \} $, where $``\cdot" $ is the usual inner product in $\R^d$.
For $x\in \R^d$, if no confusion arises we let $|x|$ be its Euclidean norm.
For $k \in \N$ we denote by $M_k(\R)$ the set of $ k \times k$ matrices with real entries,
and by $\OO (k)$ the set of $k \times k$ orthogonal matrices.

Throughout the text the letter $C$ with or without a subscript denotes an absolute constant which may vary from formula to formula.
For two quantities $a$ and $b$ we write $a \lesssim b$ if there is an absolute constant $C$ such that $a \leq C b$.
For $a,b$ depending on some parameter $\delta$, we may write $a\lesssim_\delta b$ or $a\leq C_\delta b$,
to point out that the constant in the inequality depends on $\delta$ and is otherwise absolute.

The word ``smooth'' always means differentiable of class $C^\infty$.

\section{Boundary layer systems and construction of homogenized data $g^*$}\label{sec-contr-of-g}

For a unit vector $n \in \R^d$ and scalar $a\in \R$ 
set $\Omega_{n ,a}=  \{x\in \R^d: \  x \cdot n >a \} $, 
and for a smooth and $\Z^d$-periodic vector-function $v_0$
consider the following problem
\begin{equation}\label{bdry-layer-system}
\begin{cases} -\nabla \cdot A(y) \nabla v(y) =0 ,&\text{  $y \in \Omega_{n,a}$}, \\
 v(y)=v_0(y) ,&\text{  $ y \in \partial \Omega_{n,a}  $.}    \end{cases}
\end{equation}

Problems of the form (\ref{bdry-layer-system}) will be referred to as \emph{boundary layer systems}.
These type of systems have a central role  in
the theory of periodic homogenization of Dirichlet problem for divergence type elliptic
operators with (simultaneously) oscillating coefficients and boundary data.
In a nutshell, the relevance of (\ref{bdry-layer-system}) to homogenization of \eqref{problem-osc-oper}-\eqref{prob-Dir-data-osc}  can be traced as follows.
In a small neighbourhood of a given point $x_0\in \partial D$ having normal $n$, one tries to 
attribute oscillations of $u_\e$ caused by boundary data to a new independent
variable $y$ which leads to approximating the solution $u_\e$ by a function of the form $v(x,x/\e)$
periodic in its second (oscillating) variable. Plugging such a $v$ into the equation formally, leads to a problem
of the form (2.1) where taking $\e \to 0$ amounts to asymptotics of $v$ far away from the boundary of corresponding halfspace $\Omega_n$,
which is meant to model the halfspace containing $D$ determined by tangent hyperplane of $\partial D$ at $x_0$.
Questions concerning well-posedness of boundary layer systems
and behaviour of solutions far away from the boundary of the corresponding hyperplane,
form a significant portion of the analysis toward obtaining quantitative results for homogenization of the mentioned class
of Dirichlet problems. We refer the reader to \cite{GM1}, \cite{GM2}, and \cite{Prange} for 
details concerning emergence of boundary layer systems in homogenization and their analysis.
We will however, recall the following result which is necessary for our purposes.

\begin{theorem}\label{Thm-Prange}{\normalfont(see\footnote{The current formulation is slightly different from the original
one, in that we only require $n$ to be irrational in part 2 of the Theorem. This, however, is the outcome of the original
proof, since part 1 shows that the only solution with the mentioned properties is the one given by Poisson kernel, and the only recourse to irrationality of $n$ is necessary
for the asymptotic analysis of the solution away from the boundary.} \cite[Theorem 1.2]{Prange})}
In (\ref{bdry-layer-system}) assume $A$ satisfies conditions (A1)-(A3),
$v_0\in C^{\infty}(\mathbb{T}^d; \hspace{0.05cm} \R^N)$, and let $n \in \mathbb{S}^{d-1}$. Then
\begin{itemize}
 \item[{\normalfont{1.}}] there exists a unique solution
 $v\in C^\infty(\overline{\Omega_{n,a} }) \cap L^\infty (\Omega_{n,a})$ of (\ref{bdry-layer-system}) such that
$$
\qquad || \nabla v ||_{L^\infty(\{y \cdot n >t\} ) } \to 0, \text{ as } t \to \infty ,
$$

$$
\int_a^\infty || (n \cdot \nabla) v ||_{L^\infty(\{ y\cdot n -t=0 \})  }^2 dt <\infty,
$$

\item[{\normalfont{2.}}] if in addition $n \notin \R \Q^d$, then there exists a boundary layer tail $v^\infty \in \R^N$ independent of $a$ so that 
$$
v(y) \to v^\infty, \text{ as } y\cdot n \to \infty ,
$$
and the convergence is locally uniform with respect to the tangential variables.
\end{itemize}

\end{theorem}

Now, following \cite{GM1} and \cite{Prange} we describe the construction of the homogenized boundary data.
First, consider the case when boundary data $g$ in (\ref{prob-Dir-data-osc}) can be factored into
independent components depending on $x$ and $y$. Namely, assume that there exists a smooth $v_0$ defined on $\mathbb{T}^d$
with values in $M_N(\R)$ and some smooth $g_0$ defined on $\partial D$ and with values in $\R^N$ so that
$g(x,y)=v_0(y) g_0(x)$. Next, take any $x\in \partial D$ such that $n(x) \notin \R\Q^d$,
and for $n(x)$ consider the boundary layer system (\ref{bdry-layer-system})
with boundary data $v_0$. Then let $v^\infty(x)$ be the constant
field provided by Theorem\footnote{It should be remarked that technically
Theorem \ref{Thm-Prange} is formulated for the case when the boundary data is an $N$-dimensional
vector, while here we need an $N\times N$ matrix. Clearly this is not an issue, since one may treat each column of the matrix separately,
as is mentioned e.g. in \cite{GM1}.} \ref{Thm-Prange}. Observe, that we do not need to specify the parameter $a$ in (\ref{bdry-layer-system}),
since in view of Theorem \ref{Thm-Prange} the boundary layer tail $v^\infty$ is independent of $a$ for irrational directions. Thus, without loss of generality we may assume that $a=0$.
Finally, for $x\in \partial D$ satisfying $ n(x)\notin \R \Q^d$ set
$$
g^\ast (x):= v^\infty(n(x)) g_0(x).
$$
As we have discussed above, the Gauss map of $\partial D$ realizes a diffeomorphism between
$\partial D$ and $\mathbb{S}^{d-1}$, hence $g^\ast$ is defined almost everywhere on $\partial D$.
The general case proceeds by approximation. Using periodicity of $g$ in $y$ and its smoothness
we have the following expansion
$$
g(x,y)=\sum_{\xi \in \Z^d} c_\xi (x) e^{2\pi i y \cdot \xi } = : \sum_{ \xi \in \Z^d} g_\xi (x,y),
$$
where the series converge uniformly and absolutely.
Here $g_\xi (x,y)$ is factored since $c_\xi \in \R^N$ and we may identify the exponential
$e^{2\pi i \xi \cdot y}$ with $e^{2\pi i \xi \cdot y} I_N$, where $I_N\in M_N(\R)$ is the identity matrix.
We let $v_\xi^\infty$ be the constant field corresponding to the $\xi$-th exponential.
Then, it is shown in \cite{GM1} that the homogenized boundary data is given by
\begin{equation}\label{g-as-a-series}
g^\ast(x)=   \sum\limits_{\xi \in \Z^d} c_\xi (x) v_\xi^\infty(n(x)) =: 
\sum\limits_{\xi \in \Z^d} g_\xi^\ast(x), 
\end{equation}
where $x\in \partial D$ and $n(x) \notin \R \Q^d$.
We refer the reader to Section 4.2 of \cite{GM1} for the details\footnote{In fact \cite{GM1} only treats Diophantine normals in a sense of (\ref{Diop-class}).
As we have seen above all points of $\partial D$ up to measure zero satisfy (\ref{Diop-class}) for some parameter $\kappa>0$,
and hence (\ref{g-as-a-series}) is defined almost everywhere on $\partial D$.
The extension of (\ref{g-as-a-series}) to all irrational directions follows from Theorem \ref{Thm-Prange}. }. 

A starting point of our analysis will be a representation formula for $v^\infty$
computed in \cite{Prange}, for which we need some preliminary definitions.
Recall that $A^*$ is the coefficient tensor for the adjoint operator, i.e. $(A^*)^{\alpha \beta}_{ij} = A^{\beta \alpha}_{j i}$.
Next, for all $1\leq \gamma\leq d$ we let $v^{*,\gamma}  \in M_N(\R) $ be the solution (in the sense of Theorem \ref{Thm-Prange})
to the following system
\begin{equation}\label{bdry-layer-system-for-v-star}
\begin{cases} -\nabla_{\yy} \cdot A^*(\yy) \nabla_{\yy} v^{*, \gamma }(\yy)  =0 ,&\text{  $\yy  \in \Omega_n $}, \\
 v^{*, \gamma} (\yy)=  -\chi^{*, \gamma} (\yy)   ,&\text{  $ \yy \in \partial \Omega_n  $,}    \end{cases}
\end{equation}
where $ \chi^{*, \gamma} \in M_N(\R)$ is the solution to the following \emph{cell-problem}
\begin{equation}\label{cell-problem}
\begin{cases} -\nabla_y \cdot A^*(y) \nabla_y  \chi^{*, \gamma }(y)  = \partial_{y_\alpha} A^{*, \alpha \gamma }    ,&\text{  $ y  \in  \mathbb{T}^d $}, \\
 \int_{ \mathbb{T}^d }  \chi^{*, \gamma} (y) dy =  0   .    \end{cases}
\end{equation}

We will also need a certain analogue of the notion of \emph{mean-value} for almost-periodic functions
given by the next lemma.

\begin{lem}\label{lem-quasi-periodic}{\normalfont(see \cite[Theorem S.3]{Subin})}
Let $f: \R^d \to \R$ be almost-periodic. Then there exists a scalar $\mathcal{M}(f)$ such that for any $\varphi \in L^1(\R^d)$
one has
$$
\int_{ \R^d  }  \varphi(y)  f (\lambda y)  dy \to \mathcal{M}(f) \int_{\R^d}  \varphi(y) dy, \qquad \text{as } \lambda \to \infty.
$$
\end{lem}

\noindent The following useful formula for $v^\infty(n)$
defined by Theorem \ref{Thm-Prange} is due to C. Prange (see formula (6.4) in \cite{Prange}).
Keeping the notation of Theorem \ref{Thm-Prange} and Lemma \ref{lem-quasi-periodic} we have
\begin{multline}\label{form-of-v-infty}
 v^\infty(n)= \int_{\partial \Omega_n} \partial_{y_\alpha} G^0(n, y) d\sigma(y) \times \bigg[ \mathcal{M} \{ A^{\beta \alpha}(y) v_0(y) n_\beta \} + \\
 \ \ \ \mathcal{M}\left\{ \partial_{y_\beta}(\chi^{*,\alpha})^t(y) A^{\beta \gamma}(y) v_0(y) n_\gamma \right\} + \\ 
  \mathcal{M} \left\{ \partial_{y_\beta}( v^{*,\alpha})^t (y) A^{\beta \gamma}(y) v_0(y) n_\gamma \right\} \bigg],
\end{multline}
where $\Omega_n = \{x\in \R^d: \ x \cdot  n >0 \}$
and $G^0$ is the Green's kernel corresponding to the homogenized constant coefficient operator
$-\nabla \cdot A^0 \nabla$ in domain $\Omega_n$. Also, the averages $\mathcal{M}\{\cdot\}$ are understood
for restrictions of functions on the hyperplane $\Omega_n$, that is one may apply Lemma \ref{lem-quasi-periodic}
after rotating the hyperplane $\partial \Omega_n$ to $\R^{d-1}\times \{ 0 \}$.
More precisely,
for $F:\R^d \to \R^N$ and $n\in \mathbb{S}^{d-1}$ one takes a matrix $M\in \OO(d)$
such that $M e_d = n$ and applies Lemma \ref{lem-quasi-periodic}
for a function $f(z') = F(M(z', 0))$, $z'\in \R^{d-1}$.
We do not enter into details concerning almost-periodic functions, as here our treatment will
be self-contained. The interested reader is referred, for example, to \cite{Subin} for particulars.

\vspace{0.2cm}

\noindent \textbf{The strategy of the proof.} We are now in a position to give an outline of the strategy
of the proof of Theorem \ref{thm-Main}.
From \eqref{g-as-a-series} and \eqref{form-of-v-infty} it is apparent that 
the regularity of $g^*$ depends on the regularity of $v^\infty$ with respect to the normal directions,
and we will proceed by analysing the dependence on the normal field of the quantities involved in \eqref{form-of-v-infty}.

In Section \ref{sec-Green} we show, mostly through linear algebra and some basic properties of
Green's kernel, that integrated Green's kernel in \eqref{form-of-v-infty} as a matrix-function
of $n$ is smooth on $\mathbb{S}^{d-1}$. It should be noted that we do not prove the smoothness of Green's kernel itself with respect to $n$.
That problem can be analysed using Lemma \ref{lem-smooth-rotation}
which also indicates that there are some topological objections to global smoothness of these kernels on $\mathbb{S}^{d-1}$.
Next, using Fourier-analytic approach (Lemma \ref{lem-quasi-per-my} and its corollaries) we show that $\mathcal{M}$-averages are
well-behaved for a class of almost-periodic functions. In particular that allows us to compute the first two averages
in \eqref{form-of-v-infty} explicitly. Since in general the corrector $v^{*,\alpha}$ does not fall into the realm of applicability of Lemma \ref{lem-quasi-per-my},
we analyse the last average of \eqref{form-of-v-infty}
in Section \ref{Sec-proof}  - the main part of this paper. 
It is there that assumption (A5) enters the proof, allowing us to transform the boundary layer system for $v^{*,\alpha}$ from $\Omega_n$ to $\R^d_+$
by linear change of variables, while keeping the periodicity of the operator and the boundary data intact (however, by the price of making the ellipticity constant of the operator worse).
Then, using Tartar's construction (Theorem \ref{thm-Lions-bdry-layers})
we show that the solution to \eqref{bdry-layer-system-for-v-star} has exponentially decaying gradient in the normal direction
and is periodic in tangential directions. We then use these properties to get expansion of the corrector into series of exponentials (formula \eqref{exp-for-bdry-layer-corr})
and show Lipschitz regularity of coefficients of the expansion with respect to normal directions (Lemma \ref{Lem-main}) by elliptic regularity arguments.
This enables us to apply corollaries of Lemma \ref{lem-quasi-per-my} to the last average of \eqref{form-of-v-infty} as well.
Finally, the proof of Theorem \ref{thm-Main} glues the analysis for $v^\infty$
with expansion in \eqref{g-as-a-series} to produce the result for $g^*$.

\vspace{0.2cm}

We finish this section by two observations. 
First, we compute the constant $\mathcal{M}$ for some class of almost-periodic functions,
and second, we establish 
a uniform bound on the constant field of Theorem \ref{Thm-Prange} in terms of the corresponding boundary data.

\begin{lem}\label{lem-quasi-per-my}
Let $T$ be a fixed $d\times d$ matrix with rational coefficients,
and assume we are given a function $ f(y ) = \sum\limits_{\xi \in \Z^d} c_\xi(f) e^{2\pi i T\xi \cdot y} $, $y\in \R^d$,
where each $c_\xi \in \mathbb{C}$, and $\sum\limits_{\xi \in \Z^d} |c_\xi(f)|<\infty$.
For a unit vector $n \notin \R\Q^d$ and a matrix $M\in \OO(d)$
satisfying $M e_d = n$, set $h(z') = f(M(z', 0))$, where $z' \in \R^{d-1}$. Then
$$
 \mathcal{M}(  h  )  = \sum_{\xi: \ T \xi=0} c_\xi(f).
$$
\end{lem}

\begin{proof}
To compute $\mathcal{M}(h)$ fix some
$\varphi \in C_0^\infty( \R^{d-1} )$, set $\phi_\xi (z') = T \xi \cdot M(z',0)$ for $\xi \in \Z^d$
and consider
\begin{equation}\label{def-of-I-k}
 \mathcal{I}_{ \xi } (\lambda) = \int_{\R^{d-1}} \varphi(z') e^{ -2 \pi i \lambda \phi_\xi (z') } d z', \qquad \lambda >1.
\end{equation}
The proof will be completed once we show that for each $\xi$ satisfying $T \xi \neq 0$ one has $\mathcal{I}_{ \xi } (\lambda) \to 0 $
as $\lambda \to \infty$. We henceforth assume that $T \xi \neq 0$.

It follows from the definition of the matrix $M$ that $M=[N | n]$, where $N$ is a $d\times (d-1)$ matrix.
We have $T\xi \cdot M(z', 0 )= N^t T\xi \cdot z'$, and hence
$\nabla' \phi_\xi (z') = N^t T \xi $, for all $z'\in \R^{d-1}$ where $\nabla' $ is the gradient in $\R^{d-1}$.
But as $M$ is orthogonal, it preserves the Euclidean length, consequently
$$
|T \xi | = | M^t T \xi | = | ( N^t T \xi , n\cdot T \xi  ) | = | ( \nabla' \phi_\xi (z'), n \cdot T \xi ) | .
$$
Therefore, if we assume that $\nabla' \phi_\xi (z') = 0'\in \R^{d-1}$, we get $ | n \cdot T \xi |= |T \xi| $,
which, by the equality case in Cauchy-Schwarz inequality
infers $n=T \xi /|T \xi|$.
Since $T$ has rational entries, it follows that $T\xi \in \R \Q^d$, and hence so is $n$,
contradicting the assumption that $n$ is not rational. We thus conclude that
$\nabla' \phi_\xi (z') \neq 0'$. Using this, we invoke integration by parts in (\ref{def-of-I-k})
(cf. ``the principle of the non-stationary phase" in \cite{Stein}, p. 341, Prop. 4)
and get that $\lim \limits_{\lambda \to \infty} \mathcal{I}_\xi (\lambda) = 0$,
for any $\xi \in \Z^d $ with the property $T \xi \neq 0$, which completes
the proof of the lemma.
\end{proof}

For vector-valued functions, in view of the linearity of the averaging operator $\mathcal{M}$,
and choosing matrix-valued test functions in the proof of Lemma \ref{lem-quasi-per-my},
we immediately get the following. 
\begin{cor}\label{cor-quasi-per-vector}
For $k\in \mathbb{N}$ assume $f=(f_1,...,f_k)$ where each component $f_i$ satisfies Lemma \ref{lem-quasi-per-my}.
Similarly, define $h=(h_1,...,h_k)$. Then
$$
\mathcal{M}(h) = ( \mathcal{M}(h_1),...,\mathcal{M}(h_k) ).
$$
\end{cor}

Observe that if $T$ in Lemma \ref{lem-quasi-per-my} is the identity matrix, then $f$ is $\Z^d$-periodic,
and $c_\xi$ is the $\xi$-th Fourier coefficient of $f$. 
This observation directly implies 
the independence of
the first two averages involved in the formula (\ref{form-of-v-infty}) from the normal $n \notin \R \Q^d$.
Namely, since $A$, $v_0$, and $\chi^{*,\gamma}$ are all $\Z^d$-periodic,
from Lemma \ref{lem-quasi-per-my} and Corollary \ref{cor-quasi-per-vector} we get
\begin{equation}\label{avg-stability-1}
 \mathcal{M} \{ A^{\beta \alpha}(y) v_0(y) n_\beta \} = \mathcal{M} \{ A^{\beta \alpha}(y) v_0(y)  \} n_\beta= c_0(A^{\beta \alpha} v_0 ) n_\beta,
\end{equation}
and
\begin{multline}\label{avg-stability-2}
 \mathcal{M}\left\{ \partial_{y_\beta}(\chi^{*,\alpha})^t(y) A^{\beta \gamma}(y) v_0(y) n_\gamma \right\} =
 \mathcal{M}\left\{ \partial_{y_\beta}(\chi^{*,\alpha})^t(y) A^{\beta \gamma}(y) v_0(y)  \right\} n_\gamma = \\
 c_0[ \partial_{y_\beta} (\chi^{*,\alpha})^t A^{\beta \gamma}  v_0] n_\gamma,
\end{multline}
where we have $n\notin \R \Q^d$, and $c_0(f)$ denotes the $0$-th Fourier coefficient of $\Z^d$-periodic function $f$,
i.e. the integral of $f$ over $\mathbb{T}^d$. Note that at this stage we are not able to apply Lemma \ref{lem-quasi-per-my}
to the last average in (\ref{form-of-v-infty}).

We will also need a setting when we apply $\mathcal{M}$ on a one-parameter family of functions. 
The next statement follows from Lemma \ref{lem-quasi-per-my}
in a straightforward manner.

\begin{cor}\label{cor-to-lem-quasi-per}
Let $T$, $n$, and $M$ be as in 
Lemma \ref{lem-quasi-per-my}, and let $\mathcal{E}$ be some fixed set of parameters.
Suppose for each $\tau \in \mathcal{E} $ we have a function
$ f_\tau (y ) = \sum\limits_{\xi \in \Z^d} c_\xi(f_\tau ) e^{2\pi i T\xi \cdot y} $, $y\in \R^d$,
where each $c_\xi(f_\tau) \in \mathbb{C}$,  $\sum_{\xi \in \Z^d} |c_\xi(f_\tau)|<\infty$,
and for some absolute constant $C_0$ one has
$$
|c_\xi(f_\tau) - c_\xi (f_\sigma) | \leq C_0   |\tau - \sigma|, \ \  \tau,\sigma \in \mathcal{E}  \text{ and } \xi \in \Z^d. 
$$
Then, for any $g\in C^\infty(\mathbb{T}^d)$, 
setting $h_\tau(z') = (fg)(M(z', 0))$, where $z' \in \R^{d-1}$, we get
$$
 | \mathcal{M}(  h_\tau  ) - \mathcal{M}(h_\sigma) |  \leq C_g |\sigma - \tau|, \qquad \sigma, \tau \in \mathcal{E}.
$$
\end{cor}
\begin{proof}
By $c_\xi(g)$ denote the $\xi$-th Fourier coefficient of $g$. Then by Lemma \ref{lem-quasi-per-my} we have
$$
\mathcal{M}(h_\tau) =\sum\limits_{\xi : \  T \xi\in \Z^d} c_\xi(f_\tau) c_{-T \xi} (g), \qquad \tau\in \mathcal{E}.
$$
The proof now follows by writing
$| \mathcal{M}(  h_\tau  ) - \mathcal{M}(h_\sigma) |  \leq C_0 |\sigma - \tau| \sum_{\xi\in \Z^d}   |c_\xi(g)| $,
where convergence of the series is due to the smoothness of $g$.
\end{proof}

Again, generalization to the vector-valued case is trivial.
We next proceed to a uniform estimate for the  boundary layer tail.
The claim of the next lemma follows from the Poisson representation of solutions proved in \cite{Prange}
and a bound for Poisson kernel proved in \cite{GM1}. Due to the lack of an explicit reference we include the proof here.

\begin{lem}\label{lem-bddness of boundary tail}
Keeping the assumptions and notation of Theorem \ref{Thm-Prange}, for a unit vector
$n\notin \R \Q^d$ and boundary data $v_0$ let $v^\infty$ be the corresponding constant field.
Then there exists a constant $C=C(A,d)$ independent of $n$ and $v_0$, such that $|v^\infty | \leq C || v_0 ||_{L^\infty(\T^d)} $.
\end{lem}

\begin{proof}
By \cite[Section 3.2]{Prange} for the solution of (\ref{bdry-layer-system}) one has
$$
v(y) = \int_{ \partial  \Omega_n} P(y,\widetilde{y}) v_0(\widetilde{y})d \sigma(\widetilde{y}), \qquad y\in \Omega_n,
$$
where $P$ is the Poisson kernel for (\ref{bdry-layer-system}), and satisfies the
estimate (see \cite[Lemma 2.5]{GM1})
\begin{equation}\label{Poisson-kernel}
 |P(y,  \widetilde{y} ) | \leq C \frac{y \cdot n}{| y - \widetilde{y} |^d},
\end{equation}
for all $d\geq 2$, $y\in \Omega_n$ and $\widetilde{y}\in \partial \Omega_n$,
and the constant $C$ depending on the operator and dimension $d$ only.
Using (\ref{Poisson-kernel}) one gets
$$
|v(y)| \leq C || v_0||_{L^\infty(\T^d)} \int_{ \widetilde{y} \cdot n =0  } \frac{y \cdot n}{| y - \widetilde{y} |^d} d \sigma(\widetilde{y} ) .
$$
For $M \in \OO(d)$ satisfying $n=M e_d$ make a change of variables in the last integral
by  $y=M z$ and $\widetilde{y}=M \widetilde{z}$. Due to orthogonality of $M$ we have $M^t n=e_d$, hence for any $y\in \Omega_n$ we get
$y \cdot n= z \cdot M^t n  =  z \cdot e_d  =z_d>0$ from which it follows that
$$
|v( M z)| \leq C || v_0||_{L^\infty(\T^d)} z_d  \int_{\widetilde{z}_d =0 } \frac{ d \sigma( \zz) }{ | z- \widetilde{z} |^d } = 
C \frac{ || v_0||_{L^\infty(\T^d)} }{z_d^{d-1}}  \int_{\widetilde{z}_d =0  }
\frac{ d  \sigma(\zz) }{ \left[  1+ \sum\limits_{i=1}^{d-1} \left( \frac{z_i - \zz_i}{z_d}  \right)^2   \right]^{d/2} }.
$$
Setting $ \tau_i := (z_i - \widetilde{z}_i)/z_d, \ i=1,2,...,d-1$ in the last integral, we obtain
$$
|v( M z)| \leq C  || v_0||_{L^\infty(\T^d)} \int_{\R^{d-1}} \frac{d \tau}{(1+|\tau|^2)^{d/2}} \leq C || v_0||_{L^\infty(\T^d)},
$$
finishing the proof.
\end{proof}

\section{Regularity of integrated Green's kernels with respect to normals}\label{sec-Green}

In this section we study regularity
of integrated Green's kernels in formula (\ref{form-of-v-infty}) with respect to normals $n\in \mathbb{S}^{d-1}$.
We start with some basic preliminaries.

For a coefficient tensor $A$ and a halfspace $\Omega \subset \R^d$,
the Green's kernel $G=G(y,\yy) \in M_N(\R)$ corresponding to the operator $-\nabla \cdot  A(y) \nabla$ in domain $\Omega$
is a matrix-function satisfying the following elliptic system
\begin{equation}\label{Green-main-half-space}
\begin{cases} -\nabla_{y} \cdot A(y)  \nabla_{y} G(y, \yy) =\delta(y - \yy)I_N ,&\text{  $y \in \Omega $}, \\
 G(y, \yy) =0 ,&\text{  $ y \in \partial \Omega  $,}    \end{cases}
\end{equation}
for any $\yy\in \Omega$, where $\delta$ is the Dirac distribution and $I_N\in M_N(\R)$ is the identity matrix.
To have a quick reference to this situation, we will say that $G$ is the Green's kernel for the \emph{pair} $(A, \Omega)$.
The existence and uniqueness of Green's kernels for divergence type elliptic systems in halfspaces is proved
in \cite[Theorem 5.4]{Hofmann-Kim} for $d\geq 3$, and in \cite[Theorem 2.21]{Dong-Kim} for $d=2$.
Moreover, if $A^*$ is the coefficient tensor for the adjoint operator,
and $G^*$ is the corresponding Green's kernel, then one has the following symmetry relation
\begin{equation}\label{symm-of-Green-transp}
 G^t(y,\yy)  = G^*(\yy, y), \qquad y,\yy \in \Omega.
\end{equation}

Let $B^0$ be a constant coefficient elliptic tensor and
$G^0(z, \zz)$ be the Green's kernel for the pair $(B^0, \R^d_+)$.
Fix a unit vector $n\in \Ss^{d-1}$, along with a matrix $M\in \OO(d)$ satisfying $M e_d = n$.
Note, that we have no assumption on $n$ being a rational or an irrational direction.
For $y,\yy \in \Omega_n$ set $G^{n} ( y, \yy):= G^0 (M^t y , M^t \yy) $,
we now determine a system of equations satisfied by the matrix $G^{n}$.

Clearly, for any $y \in \partial \Omega_n$ one has $M^t y \in \partial \R^d_+$
and hence $G^{n}(y, \yy) = 0$, so we get a zero boundary condition for $G^{n}$ in $\Omega_n$ 
for any $\yy \in \Omega_n$. To get the system for $G^n$, let us rewrite the system in the definition of
the Green's kernel in (\ref{Green-main-half-space}). Let $G^0=(G^0_{kj} ) \in M_N(\R)$, then
according to (\ref{Green-main-half-space}) for all $1\leq i,k \leq N,$ we have
\begin{equation}\label{Green-explicit}
-\partial_{z_\alpha} ( B^{0,\alpha \beta}_{ij} \partial_{z_\beta} G^0_{kj} (z,\zz)  ) =\delta(z-\zz) \delta_{ik}  , \qquad z\in \R^d_+,
\end{equation}
where $\delta_{ik}$ is the Kronecker delta. For fixed $1\leq i,j\leq N$ denote $B^0_{ij}: = (B^{0,\alpha \beta}_{ij}) \in M_d(\R)$, then with this notation
(\ref{Green-explicit}) becomes
$$
-\nabla_z \cdot B^0_{ij} \nabla_z G^0_{kj} (z,\zz) = \delta(z-\zz) \delta_{ik}.
$$
Now, fix $\yy\in \Omega_n$, then for any $1\leq \alpha \leq d$ we have
$$
\partial_{y_\alpha} G^n_{kj}(y,\yy) = \partial_{z_1} G^0_{kj}(M^t y,M^t \yy) m_{\alpha 1}+...
+\partial_{z_d} G^0_{kj}(M^t y,M^t \yy) m_{\alpha d},
$$
and hence $ \nabla_y G^n_{kj}(y,\yy) = M \nabla_z G^0_{kj} ( M^t y, M^t \yy )  $,
from which we obtain
\begin{equation}\label{Green-change-of-var}
\nabla_y \cdot B^0_{ij} \nabla_y G^n_{kj} (y,\yy) = \nabla_z \cdot M^t B^0_{ij} M \nabla_z G^0_{kj} (z,\zz),
\end{equation}
where $z=M^t y$ and $\zz = M^t \yy$. Observe that by non-degeneracy of $M$ we have
$\delta(z-\zz) = \delta(M^t (y-\yy) ) = \delta(y-\yy)$,
which in combination with (\ref{Green-change-of-var}) implies the following.

\begin{claim}\label{claim-Green-relat}
Let $n\in \mathbb{S}^{d-1}$ be any, and $M\in \OO(d)$ be such that $M e_d = n$.
If $G^{0,n}(z,\zz)$ is the Green's kernel for the pair $(M^t B^0 M, \R^d_+ )$,
then $G^n(y,\yy):= G^{0,n} (M^t y, M^t \yy)$ is the Green's kernel for the pair $(B^0, \Omega_n)$,
where $M^t B^0 M$ is understood in accordance with (\ref{Green-change-of-var}).
\end{claim}

Now let $G^n(y,\yy)$ be the Green's kernel for the pair $(A^0, \Omega_n)$,
where $A^0$ is the homogenized tensor corresponding to $A(y)$. For $1\leq \alpha \leq d$, set
\begin{equation}\label{int-of-Green}
\mathcal{I}^\alpha (n) = \int_{\partial \Omega_n} \partial_{\yy_\alpha} G^n(n, \yy) d \sigma(\yy),
\end{equation}
which is precisely the term involved in the formula (\ref{form-of-v-infty}). 
Let us stress that $\mathcal{I}^\alpha (n)$ is well-defined for any $n \in \mathbb{S}^{d-1}$ and
the goal is to establish regularity of $\mathcal{I}^\alpha$ as a function from the unit sphere $\Ss^{d-1}$
to the space of matrices $M_N(\R)$ which, for this purpose, is identified with $\R^{N^2}$ in a usual manner.
Let $G^{0,n}(z, \zz)$ be the Green's kernel for the pair $(M^t A^0 M, \R^d_+)$, 
then by Claim \ref{claim-Green-relat} and the computations preceding that we have
$$
\partial_{\yy_\alpha} G^n_{jk}( n, \yy) = \partial_{ \zz_1 }  G^{0,n}_{jk} (e_d, M^t \yy) m_{\alpha 1} + ... + 
\partial_{ \zz_d }  G^{0,n}_{jk} (e_d, M^t \yy) m_{\alpha d}.
$$
Using this we make a change of variables in (\ref{int-of-Green}) by the formula $\yy = M \zz$, where $\zz \in \R^d_+$.
As $G^{0,n}$ has zero boundary conditions with respect to both variables,
we get that all tangential derivatives
in the last expression are vanishing. Also, since $M e_d = n$ it follows that $m_{\alpha d} = n_\alpha$ for any $1\leq \alpha \leq d$.
We thus get
\begin{equation}\label{green-int-stand}
\mathcal{I}^\alpha (n) = n_{\alpha} \int_{  \partial \R^d_+ } \partial_{\zz_d  } G^{0,n} (e_d, \zz) d \sigma(\zz).
\end{equation}

The following bound is proved in \cite[estimate (2.17) of Lemma 2.5]{GM1}
$$
 | G^{0,n}(z,\zz) | \leq C \frac{z_d \zz_d}{|z-\zz|^d} , \qquad z\neq \zz \text{ in } \R^d_+,
$$
where $C$ is independent of $n$. Since $G^{0,n}(e_d, \cdot)$ is zero on $\partial \R^d_+$, from the last estimate it easily
follows that $ | \nabla_{\zz} G^{0,n} (e_d,\zz)  |\leq C  |e_d-\zz|^{-d}  $,
for all $\zz \in \partial \R^d_+$, and hence the integral in (\ref{green-int-stand}) is absolutely convergent,
and is uniformly bounded with respect to $n$. 

\begin{remark}
Observe, that while $\mathcal{I}^\alpha (n)$ is independent
of the orthogonal matrix $M$, the kernel $G^{0,n}(z, \zz)$ implicitly depends on $M$.
For the objective of this section the choice of $M$ is irrelevant, and for the clarity of notation
we do not incorporate it into the notation for $G^{0,n}$. However, in the analysis of
regularity of kernels $G^{0,n}$ with respect to $n$, the choice of
$M$ plays a key role. The choice of rotation matrices is discussed in subsection \ref{sub-rot-Lip}.
It is interesting to observe, that whereas the integral of $G^{0,n}$ is 
easily seen to be smooth with respect to $n$, proving a similar result for $G^{0,n}$ itself is
comparatively more involved, and contains some topological nuances briefly discussed in the Appendix.
\end{remark}

We finish this section with the following result.

\begin{lem}\label{Lem-Green-reg-systems}
For any $1\leq \alpha \leq d $ each component of the matrix function $\mathcal{I}^\alpha (n): \mathbb{S}^{d-1} \to M_N(\R)$
is a smooth real-valued function on $\mathbb{S}^{d-1}$.
\end{lem}
\begin{proof}
Set $\mathcal{I}(n) = \int_{  \partial \R^d_+ } \partial_{\zz_d  } G^{0,n} (e_d, \zz) d \sigma(\zz)$,
clearly it is enough to prove the claim for the matrix-function $\mathcal{I}(n)$.
In view of Claim  \ref{claim-Green-relat} the coefficient tensor corresponding to $G^{0,n} $ is $M^t A^0 M =: B^0$. 
Next, referring to \cite[p. 358]{Prange}, we know that
the Poisson's kernel $P^{0,n}=(P^{0,n}_{ij})_{i,j=1}^N  $ corresponding to $G^{0,n}$ is defined by
$$
P^{0,n}_{ij} (z,\zz) = -B^{0,\alpha \beta}_{k j} \partial_{\zz_\alpha} G^{0,n}_{i k } (z, \zz) (e_d)_\beta, \qquad z\in \R^d_+, \ 
\zz \in \partial \R^d_+.
$$
Since $G^{0,n}$ has zero boundary conditions in $\R^d_+$ with respect to both
of its variables, all tangential derivatives in the last expression are vanishing,
and as $(e_d)_\beta = \delta_{\beta d}$, for all $1\leq i,j\leq N$ we obtain
\begin{equation}\label{Poisson-reduced}
P^{0,n}_{ij}(z, \zz) =  - B^{0,d d}_{k j} \partial_{\zz_d } G^{0,n}_{i k } (z, \zz). 
\end{equation}
From definitions of $B^0$ and $M$, for each fixed $1\leq k , j \leq N$ we have 
$$
B^{0,d d}_{k j} = e_d^t B_{k j} e_d = e_d^t M^t A^0_{kj} M e_d = (M e_d)^t A^0_{kj} (M e_d ) = n^t A^0_{kj} n \in \R.
$$ 
Combining this with (\ref{Poisson-reduced}), for the $(i,j)$-th entry of the matrix $P^{0,n}$ we get
\begin{equation}\label{Poisson-reduced-2}
 P^{0,n}_{ij}(z, \zz) = -n^t A^0_{kj} n \partial_{\zz_d } G^{0,n}_{i k } (z, \zz).
\end{equation}
For $n\in \mathbb{S}^{d-1}$ consider the matrix $A(n) = (a_{kj}(n))_{k,j=1}^N$, where we have set $a_{kj}(n) = -n^t A^0_{kj} n$.
Now, observe that for column-vector $v_i=(0,...,1,...0)^t \in \R^N$ with $1$ on the $i$-th position, and 0 otherwise, 
we have $\int_{\partial \R^d_+} P^{0,n} (e_d, \zz ) v_i d\sigma(\zz) = v_i$ for all $n\in \mathbb{S}^{d-1}$, and any $1\leq i \leq N$.
This follows from that fact that the unique smooth solution to Dirichlet problem has Poisson integral representation.
From here and (\ref{Poisson-reduced-2}) we get
$$
I_N = \int_{\partial \R^d_+} P^{0,n}(e_d, \zz) d\sigma(\zz) = \mathcal{I}(n) A(n),
$$
where as before $I_N$ is the $N\times N$ identity matrix. It follows that the matrix $A(n)$ is invertible for any $n\in \mathbb{S}^{d-1}$,
and hence $\mathcal{I}(n) = (A(n))^{-1}$. 
On the other hand all components of $A(n)$ are obviously smooth functions on $\mathbb{S}^{d-1}$,
therefore the determinant of $A(n)$ stays away from 0 by compactness of $\mathbb{S}^{d-1}$.
We conclude that each component of the inverse $(A(n))^{-1}$ is $C^\infty$ on $\mathbb{S}^{d-1}$, hence we get the claim for $\mathcal{I}$ and finish the proof of the lemma.
\end{proof}


\section{The regularity of $g^*$}\label{Sec-proof}

The aim of this section is to prove Theorem \ref{Thm-our2}.
Observe, that so far we had no recourse to assumption (A5)
regarding the layered structure, and it is here that it
will play a central role in the analysis.

\subsection{Change of variables}\label{subsec-algebra}
At several places in this section we will switch from one variable to another;
we record the necessary details here.
Let $y\in \R^d$ and for a coefficient tensor $B=B^{\alpha \beta}(y) \in M_N(\R)$
which is smooth and elliptic in a sense of standard assumptions (A2) and (A3)
of Section \ref{sec-Intro} consider the operator
$\mathcal{L}=-\nabla_y \cdot  B(y) \nabla_y $.
For $x\in \R^d$ set $y=Tx$, where $T\in M_d(\R)$ and has non-zero determinant.
One may easily deduce that 
\begin{equation}\label{deriv-x-y-matrix-T}
\nabla_y  = (T^t)^{-1}\nabla_x =(T^{-1})^t \nabla_x.
\end{equation}

For $1\leq i,j \leq N$ let $B_{ij}$ be the $d\times d$ matrix formed from the $(i,j)$-th
entries of the matrices $B^{ \alpha \beta}$. Then using (\ref{deriv-x-y-matrix-T})
we see that the operator $\mathcal{L}$ in the new variable $x$ can be written as
$\mathcal{L}= -\nabla_x \cdot  \widetilde{B}(Tx) \nabla_x  $,
where correspondingly
\begin{equation}\label{relation-of-A-B}
\widetilde{B}_{ij}(Tx) = T^{-1} B_{ij} (Tx) ( T^{-1} )^t
\end{equation}
for all $1\leq i,j \leq N$. To keep track of the ellipticity constant of the new operator we take a family of vectors $\xi =\xi^\alpha \in \R^N$,
set $\omega_i =( \xi_i^1,..., \xi_i^d)^t $ where $1\leq i \leq N$ and compute
\begin{multline}\label{ell-constant-B}
\widetilde{B}_{ij}^{\alpha \beta} \xi_j^\beta \xi_i^\alpha = \omega_i^t \widetilde{B}_{ij} \omega_j = \omega_i^t T^{-1} B_{ij} (T^{-1})^t \omega_j =
 [ (T^{-1})^t \omega_i  ]^t B_{ij} ( T^{-1} )^t \omega_j \geq \\ \lambda_B  \sum\limits_{i=1}^N || (T^{-1})^t \omega_i ||^2 \geq
\lambda_B \sigma_{\min}^2(T^{-1}) \omega_i \cdot \omega_i = \lambda_B \sigma_{\min}^2(T^{-1}) \xi^\alpha \cdot \xi^\alpha,
\end{multline}
where $\lambda_B$ is the ellipticity constant of the original operator
and $\sigma_{\min}(T^{-1})$ is the least singular value of the matrix $T^{-1}  $, that is
the square root of the smallest eigenvalue of $T^{-1} (T^{-1})^t$.
In particular, it follows that the new operator is elliptic,
with possibly a different ellipticity constant.

\subsection{Solutions with exponentially decaying gradients}\label{subsec-exp-decaying-sols}

Here we use assumption (A5) to gain some extra control on
solutions to boundary layer systems. To illustrate
what one can get from (A5) we will start with a simple example
involving the Laplace operator.

\begin{example}\label{example-Laplace}
{\normalfont
Assume $N=1$, i.e we have only one equation,
and for an \emph{irrational} direction $n\in \mathbb{S}^{d-1}$ and $u_0 \in C^\infty(\mathbb{T}^d)$ consider the following problem
\begin{equation}\label{Laplace-in-example}
\Delta u = 0 \text{ in } \Omega_n \qquad \text{ and } \qquad u=u_0 \text{ on } \partial \Omega_n.
\end{equation}
Let $\{c_\xi(u_0)\}_{\xi\in \Z^d}$ be the sequence of Fourier coefficients of $u_0$.
Then, by a direct computation one can easily check that the function
\begin{equation}\label{Laplace-in-example-solution}
u(y) = \sum_{\xi \in \Z^d} c_\xi(u_0) e^{ -2\pi \big[ | \xi |^2 - (n \cdot \xi)^2 \big]^{\frac 12} (y\cdot n) } 
e^{ 2\pi i \xi \cdot [ y- n(y\cdot n) ] }, \qquad y \in \overline{\Omega_n},
\end{equation}
solves (\ref{Laplace-in-example}) and satisfies all requirements of Theorem \ref{Thm-Prange},
where as before $\Omega_n = \{ x\in \R^d: \ x\cdot n>0 \}$.
It follows in particular that $u$ defined by (\ref{Laplace-in-example-solution}) is the unique
solution of (\ref{Laplace-in-example}) given by Theorem \ref{Thm-Prange}.
Since $n \notin \R \Q^d$, the equality case
of the Cauchy-Schwarz inequality
provides $| \xi |^2 - (n \cdot \xi)^2 \neq 0 $ unless $\xi=0$ 
and hence the boundary layer tail in this case is simply $c_0 (u_0)$.

Now assume that $u_0$ is independent of the last coordinate, i.e. $(e_d\cdot \nabla ) u_0 =0$ on $\mathbb{T}^d$.
This condition can be reformulated in terms of Fourier coefficients.
Namely, using the smoothness of $u_0$ and applying $e_d\cdot \nabla $
on the Fourier series of $u_0$, by Parseval's identity we obtain that $\xi_d c_\xi(u_0)=0$ for all $\xi \in \Z^d$.
The latter implies that $c_\xi(u_0)=0$ for any $\xi \in \Z^d$ with $\xi_d \neq 0$,
that is the Fourier spectrum of $u_0$ is contained in the sublattice $\Z^{d-1}\times \{0\} \subset \Z^d$.
Next, suppose the vector $n$ satisfies
$n_d \neq 0$. Then
for $\xi=(\xi', 0) \in \Z^{d-1} \times \{0\}$ we have
$$
(n\cdot \xi)^2   \leq (n_1^2+...+n_{d-1}^2 ) | \xi'|^2 = (1-n_d^2) | \xi|^2,
$$
therefore
$$
\big| |\xi|^2  - (n\cdot \xi)^2 \big| = |\xi|^2 \left| 1- \frac{ (n\cdot \xi)^2 }{ |\xi|^2 } \right| \geq n_d^2 |\xi|^2.
$$
The latter combined with (\ref{Laplace-in-example-solution}) illustrates that
given the special structure of the Fourier spectrum of $u_0$, the solution of (\ref{Laplace-in-example})
converges exponentially fast in the direction of the normal vector $n$ toward its boundary layer tail.
Also, it is clear that the decay properties deteriorate as $n_d\to 0$. It should also be noted
that while $u_0$ was independent of $e_d$, the solution $u$ does not necessarily satisfy
this independence criterion.
}
\end{example}

To treat the general case we will need a construction due to L. Tartar.
For $\mathrm{Y'}$, an open parallelepiped in $\R^{d-1}$, 
set $G=\mathrm{Y'} \times (0,\infty)$.
Let $f=\{f_i\}$ and $F=\{F_i^\alpha\}$ be given smooth functions,
where $ 1 \leq i \leq N $ and $1 \leq \alpha \leq d$.
For the unknown vector $u=(u_1,...,u_N)$ consider the following problem
\begin{equation}\label{app-bdry-layer-prob}
\begin{cases}  -\nabla \cdot A(y) \nabla u(y)  = f- \nabla \cdot F(y)  ,&\text{ in $ G  $}, \\
 u(y',0)= 0 ,&\text{  $ y' \in \mathrm{Y'}  $}, \\
 u(\cdot, y_d) ,&\text{ is $\mathrm{Y'}$-periodic for any } y_d>0,     \end{cases}
\end{equation}
where the system of equations is understood as follows
$$
-\frac{\partial}{\partial y_\alpha} \left( A^{\alpha \beta}_{ij} (y) \frac{\partial u_j}{\partial y_\beta}(y) \right) 
=f_i -  \frac{ \partial F_i^\alpha }{\partial y_\alpha}, \qquad i=1,2,...,N.
$$
We assume that there exists $\tau_0>0$ such that
\begin{equation}\label{app-f-exp-decay}
 e^{\tau_0 y_d} f(y) \in L^2(G; \hspace{0.05cm} \R^N)   \qquad \text{ and }  \qquad e^{\tau_0 y_d} F(y) \in L^2(G; \hspace{0.05cm} \R^{d\times N}),
\end{equation}
and
\begin{equation}\label{app-F-vector-field-assump}
f(\cdot, y_d) \text{ and }  F(\cdot , y_d) \text{ are both $\mathrm{Y'}$-periodic for any } y_d \geq 0.
\end{equation}
In order to clarify the periodicity condition in (\ref{app-bdry-layer-prob}),
recall the definition of $H^1_{per}(\mathrm{Y'})$, 
which is the closure with respect to $H^1$-norm of the space of smooth and $\mathrm{Y'}$-periodic functions.
In particular, functions in $H^1_{per}(\mathrm{Y'})$ 
have equal traces on opposite faces of $\mathrm{Y'}$.
Now for $\tau>0$ set  
$$
V_\tau (G)=\{ v\in H^1_{loc}(G): \ v\in L^2_{loc}(\R_+; \hspace{0.05cm} H^1_{per}(\mathrm{Y'}) ), \ e^{\tau y_d } \nabla v(y) \in L^2(G) \text{ and } v(y',0)\equiv 0 \}.
$$
One can see that $V_\tau$ is a Hilbert space with scalar product defined by
$$
[u,v]_\tau =\int_G e^{2\tau y_d} \nabla u(y) \cdot \nabla v(y) dy.
$$
The norm on $V_\tau(G)$ induced from the scalar product is denoted by $|| \cdot ||_{V_\tau(G)}$.
The existence of solutions to (\ref{app-bdry-layer-prob}) with exponentially decaying gradients is given
in the following result.
\begin{theorem}\label{thm-Lions-bdry-layers}{\normalfont(see \cite[Chapter 18]{Tartar}, and \cite[Theorem 10.1]{Lions})}
Assume (\ref{app-f-exp-decay}), (\ref{app-F-vector-field-assump})
and that the coefficient tensor in (\ref{app-bdry-layer-prob})
is bounded and is uniformly elliptic with ellipticity constant $\lambda_A>0$. Then
for any $0<\tau < \min\{\tau_0,  \frac{\lambda_A}{2||A||_\infty}  \}$
there exists a unique solution $u$ to system (\ref{app-bdry-layer-prob}) in the space $V_\tau(G)$.
Moreover, for any such $\tau$ one has the estimate
\begin{equation}\label{app-thm-Lions-norm-est}
||  u ||_{ V_\tau(G) } \leq \frac{C}{\tau} \frac{1}{\lambda_A - 2 \tau ||A||_\infty}
\big[ || e^{\tau y_d} f ||_{L^2(G; \hspace{0.05cm} \R^N) } +   ||  e^{\tau y_d}  F ||_{L^2(G; \hspace{0.05cm} \R^{d\times N}) } \big],
\end{equation}
where the constant $C$ depends on dimension $d$ and the parallelepiped $\mathrm{Y'}$.
\end{theorem}

\noindent 
Observe, that at this stage we do not use periodicity of $A$, nor any other structural restriction
is imposed on the operator.

\begin{remark} 
The formulation of Theorem \ref{thm-Lions-bdry-layers}
is slightly more general than the original one as given e.g. in \cite{Lions} or \cite{Tartar}.
Namely, here it is stated for elliptic systems rather than scalar equations, 
and involves detailed estimates of $V_\tau$ norms of solutions. The proof however,
follows the lines of the original proof with small changes to deal with systems of equations, 
and making the norm estimate of $u$ explicit.
\end{remark}

\noindent The following useful fact follows directly from Theorem \ref{thm-Lions-bdry-layers}.

\begin{cor}\label{cor-Lions}
Assume the coefficient tensor $A(y)$ is bounded, uniformly elliptic
with ellipticity constant $\lambda_A>0$, 
smooth and $\mathrm{Y'}$-periodic for each fixed $y_d \geq 0$.
Then, for any smooth and $\mathrm{Y'}$-periodic vector-function $g$ with values in $\R^N$
the following problem
\begin{equation}\label{app-bdry-layer-prob-bdry-data}
\begin{cases} -\nabla \cdot A(y) \nabla u(y) =0 ,&\text{$ y\in \R^d_+  $}, \\
 u(y',0)= g(y') ,&\text{$ y' \in \R^{d-1}  $}    \end{cases}
\end{equation}
has a unique weak solution $u\in H^1_{loc}(\R^d_+)$ with the properties
\begin{equation}\label{prop-1}
u \in L^2_{loc} (\R_+; \hspace{0.05cm} H^1_{per}( \mathrm{Y'} ) )
\ \text{ and } \ e^{\tau y_d } \nabla u \in L^2( \mathrm{Y'} \times \R^d_+ ) \text{ for any }
0<\tau<\frac{\lambda_A}{2||A||_\infty}.
\end{equation}
Moreover, the solution $u$ satisfies
\begin{equation}\label{app-thm-bdry-layer-bdry-data-est}
|| e^{\tau y_d } \nabla u  ||_{ L^2(G; \hspace{0.05cm} \R^{d\times N}) }  \leq \frac{C}{\tau} \frac{ ||A||_\infty }{\lambda_A - 2\tau || A||_\infty} || g||_{H^1( \mathrm{Y'}; \hspace{0.05cm} \R^N)}.
\end{equation}
\end{cor}

\begin{proof}
Fix any non-negative and compactly supported smooth function $\varphi: \R \to [0,1]$
such that $\varphi = 1 $ near 0.
Using the cut-off $\varphi$ we lift the boundary data $g$ into $\R^d_+$
by setting $\widetilde{g}(y)=\varphi(y_d) g(y')$ for all $y=(y',y_d)\in G$.
Since $\widetilde{g}$ has compact support in the direction of $e_d$ we
have that $e^{\tau y_d} \nabla \widetilde{g}\in L^2(G; \hspace{0.05cm} \R^{d\times N})$, for any $\tau>0$,
and we let $\widetilde{u}$ be the unique solution to (\ref{app-bdry-layer-prob}) with the right-hand side $\nabla \cdot A(y) \nabla \widetilde{g} $
given by Theorem \ref{thm-Lions-bdry-layers}.
Next, we denote by $\widetilde{u}_{per}$
the extension of $\widetilde{u}$ to $\R^d_+$ by periodicity in tangential variables.
More precisely for any $(y', y_d) \in \R^d_+$ we set 
$\widetilde{u}_{per}(y',y_d) = \widetilde{u} (y' - \xi', y_d)$  where $\xi'$ is the unique element of $\Z^{d-1}$ with the property that $y' - \xi' \in \mathrm{Y'}$.
It then follows by standard arguments
that we have $\widetilde{u}_{per}\in H^1_{loc}(\R^d_+)$ for the extension and that $\widetilde{u}_{per}$
defines a weak solution to\footnote{For reader's convenience we briefly sketch the argument. First, the inclusion $\widetilde{u}_{per}\in H^1_{loc}(\R^d_+)$ is a direct corollary
to the fact that $\widetilde{u}(\cdot, y_d) \in H^1_{per}(\mathrm{Y'})$ for any $y_d \geq 0$
(see e.g. \cite[Proposition 3.50]{CD} for a similar treatment).
Next, to see that $\widetilde{u}_{per}$ solves (\ref{prob-per})
it is enough to see that $\widetilde{u}_{per}$ defines a solution across lateral boundary
of $G$, i.e. $\omega := \partial \mathrm{Y'} \times \R_+$.
Writing the definition of weak solution to \eqref{prob-per} (i.e. testing the equation  against $C^\infty_0$ functions)
we see that it suffices to have $\widetilde{u} \in H^2$ locally in a neighbourhood of each point of $\omega$
(so that to make sense of the trace of the derivatives of $\widetilde{u}_{per}$)
and $\nabla \widetilde{u} (\cdot, y_d) \in H^1_{per}(\mathrm{Y'})$ for any $y_d \geq 0$,
as then equality in (\ref{prob-per}) will simply follow by localizing the equation in a neighbourhood of $\omega$ and doing partial integration
in the weak (integral) formulation.
For the $H^2$-regularity, we first see that
tangential derivatives of $\widetilde{u}$ solve a similar problem in $V_\tau(G)$ as
$\widetilde{u}$ itself 
(by considering difference quotients instead to be more precise), which shows that tangential derivatives of $u$ have the desired
regularity and periodicity properties. After having treated the tangential derivatives, the $d$-th derivative of $\widetilde{u}$  
can be handled from the system itself, by separating the term with $dd$-th
derivative, and treating the rest as lower-order terms.
Namely, one can write $A^{dd}_{ij}  \partial^2_{dd} \widetilde{u}_j  \in L^2_{loc}( \R_+; \hspace{0.05cm} H^1_{per}( \mathrm{Y'} )  )$,
where $\widetilde{u} = (\widetilde{u}_1,...,\widetilde{u}_N)$,
and then invert the $N\times N$ matrix on the left-hand side (relying on ellipticity of $A$)
to get the mentioned regularity and periodicity properties of $\partial_d \widetilde{u}$
(cf. \cite[eq. (2.12)]{GM2}, where the situation is more complicated due to the lack 
of uniform ellipticity).}
\begin{equation}\label{prob-per}
-\nabla \cdot A \nabla \widetilde{u}_{per} = \nabla \cdot A \nabla \widetilde{g}  \text{ in } \R^d_+ \ \ 
\text{ and } \ \ \widetilde{u}_{per} = 0 \text{ on } \partial \R^d_+.
\end{equation}
Since $\widetilde{u}_{per}$ solves \eqref{prob-per}
the function $u =\widetilde{u}_{per}+\widetilde{g}$ satisfies all
requirements of the corollary, and the
estimate (\ref{app-thm-bdry-layer-bdry-data-est})
follows easily from the corresponding estimate of Theorem \ref{thm-Lions-bdry-layers}.
\end{proof}

\subsection{The case when $\nu_0 = e_d$}\label{e-d-is-nu-0}
We will first carry out the analysis when the vector $\nu_0$ defined from assumption
(A5) coincides with $e_d=(0,...,0,1) \in \R^d$.
To fix the ideas, we let $A=A^{\alpha \beta} \in M_N(\R)$ be a coefficient tensor satisfying
the standard ellipticity, smoothness, and periodicity conditions 
of Section \ref{sec-Intro}, and in addition we require $A$ to be independent of $e_d$,
or equivalently the $d$-th coordinate.
We also fix $v_0 \in C^\infty(\mathbb{T}^d)$ which is assumed to be independent of $e_d$ as well.
Then, for a given $n \in \mathbb{S}^{d-1}$ consider the following problem
\begin{equation}\label{oper-when-nu0-is-e-d}
 \begin{cases} -\nabla_y \cdot A(y) \nabla_y v(y) =0 ,&\text{  $y \in \Omega_n$}, \\
 v(y)=v_0(y) ,&\text{  $ y \in \partial \Omega_n  $.}    \end{cases}
\end{equation}
Let $v$ be the unique solution to (\ref{oper-when-nu0-is-e-d})
given by Theorem \ref{Thm-Prange}.
The aim now is to show that this solution has some extra regularity 
properties given the structural restriction on $A$ and $v_0$.

We will assume that $n_d \neq 0$, and then without loss of generality will take $n_d>0$,
as the case $n_d<0$ works in the same way. The case of $n_d=0$ is degenerate, and the analysis breaks down.
Also, notice that at this stage we do not require $n$ to be irrational.
To the unit vector $n=(n_1,...,n_d)$ we attach a matrix $T_n \in M_d(\R)$
given by
\begin{equation}\label{matrix-T-n}
T_n=\left(
      \begin{array}{cccc}
         &  &  & 0 \\
         & \text{\Large $I_{d-1}$} &  & \vdots \\
         &  &  & 0 \\
        -\frac{n_1}{n_d}  &  ...  &  -\frac{n_{d-1}}{n_d} & 1 \\
      \end{array}
    \right),
\end{equation}
where $I_{d-1}\in M_{d-1}(\R)$ is the identity.
It is clear that a linear transformation associated with $ T_n  $ is a bijection from $\overline{\R^d_+}$ to $\overline{\Omega_n}$,
and that
\begin{equation}\label{matrix-T-inverse}
T_n^{-1}=   \left(
      \begin{array}{cccc}
         &  &  & 0 \\
         & \text{\Large $I_{d-1}$} &  & \vdots \\
         &  &  & 0 \\
        \frac{n_1}{n_d}  &  ...  &  \frac{n_{d-1}}{n_d} & 1 \\
      \end{array}
    \right)
\end{equation}
is the inverse of $T_n$. We make a change of variables in (\ref{oper-when-nu0-is-e-d})
by setting $y=T_n z$, where $z \in \R^d_+$. Following the notation and results
of Section \ref{subsec-algebra},
if we let $A_n$ be the coefficient tensor in the new variable $z$ then
\begin{equation}\label{coeff-A-n}
 (A_n)_{ij}(T_n z) = T_n^{-1} A_{ij} (T_n z) (T_n^{-1})^t =T_n^{-1} A_{ij} ( z) (T_n^{-1})^t.
\end{equation}
The last equality of (\ref{coeff-A-n}) follows from the fact that the linear transformation $T_n$ acts as an identity
on the first $d-1$ variables, and affects only the $d$-th coordinate on which $A$ has no dependence by assumption.
A similar reasoning applied to $v_0$ gives $v_0 (T_n z ) = v_0 (z) $. In particular, the change of variable by $T_n$ leaves periodicity of the operator
and the boundary data invariant.
Thus, the problem (\ref{oper-when-nu0-is-e-d}) is being transformed to
\begin{equation}\label{bdry-layer-system-halfspace2}
\begin{cases} -\nabla_z \cdot A_n(z) \nabla_z w(z) =0 ,&\text{  $ z \in \R^d_+ $}, \\
 w(z)=v_0(z) ,&\text{  $ z \in \partial \R^d_+   $},   \end{cases}
\end{equation}
where $A_n$ is given by (\ref{coeff-A-n}) and $w(z)= v(T_n z)$.
The ellipticity of $A_n$ follows from non-degeneracy of $T_n$ and (\ref{ell-constant-B}).
We now give an estimate on the ellipticity constant of $A_n$ which we will use in the sequel.
Following (\ref{ell-constant-B}) we need to bound the smallest singular value of $T_n^{-1}$ from below,
which is being done using the following result.

\begin{theorem}\label{sec-app-sing-value}{\normalfont{(see \cite[Theorem 1]{H-P})}}
For a matrix $T \in M_d(\mathbb{C})$ let $r_i(T)$ be the Euclidean norm of its $i$-th row,
$c_i(T)$ be the Euclidean norm of its $i$-th column, and set
$r_{\min}(T)=\min\limits_{1\leq i \leq d} r_i(T) $ and $c_{\min}(A)=\min\limits_{1\leq i \leq d} c_i(T) $.
Then, for $\sigma_{\min}(T)$, the smallest singular value of $T$, one has
\begin{equation}\label{app-smallest-sing-value}
\sigma_{\min}(T) \geq \left( \frac{d-1}{d} \right)^{(d-1)/2} |\mathrm{det} T| \max\left\{ \frac{c_{\min}(T)}{ \Pi_{i=1}^d  c_i(T) }, 
\frac{r_{\min}(T)}{ \Pi_{i=1}^d  r_i(T) }  \right\}.
\end{equation}
\end{theorem}

We have $\mathrm{det} (T_n^{-1}) = 1 $ and using the fact that $|n|=1$
we obtain $r_{\min(T_n^{-1}) }/ \Pi_{i=1}^{d} r_i( T_n^{-1} )   = n_d  $. Now by virtue of (\ref{app-smallest-sing-value})
it follows that
\begin{equation}\label{T-n-sing-value}
\sigma_{\min}(T_n^{-1}) \geq \left( \frac{d-1}{d} \right)^{(d-1)/2} n_d.
\end{equation}
Hence, for $\lambda_{A_n}$, the ellipticity constant of the operator in (\ref{bdry-layer-system-halfspace2}),
we have by (\ref{ell-constant-B}) and (\ref{T-n-sing-value}) that
\begin{equation}\label{ell-const-of-A-n}
 \lambda_{A_n} \geq c_d \lambda_A n_d^2,
\end{equation}
where $c_d$ is a constant depending on the dimension, and $\lambda_A$ is the ellipticity constant of the original operator.
Invoking Corollary \ref{cor-Lions}
we let $w_n$ be the unique weak solution of (\ref{bdry-layer-system-halfspace2})
with finite $|| \cdot ||_{V_\tau(G)}$-norm for all $0<\tau< \frac{\lambda_{A_n}}{2 ||A_n||_\infty}$
where $G=(0,1)^{d-1}\times \R_+$.
Since the coefficients and the boundary data are smooth in
(\ref{bdry-layer-system-halfspace2}) it follows from the standard elliptic
regularity that $w_n \in C^\infty( \overline{\R^d_+}) $ (see e.g. \cite[Corollary 4.12]{Giaq} and \cite[Theorem 5.21]{Giaq}).
Moreover, we have by construction that $w_n(\cdot, z_d)$ is $ \Z^{d-1} $-periodic for each
$z_d \geq 0$ and has exponentially decaying gradient in the direction of $e_d$.
As $w_n$ solves (\ref{bdry-layer-system-halfspace2}) it follows that 
$v_n(y)=w_n(T_n^{-1} y)$ solves (\ref{oper-when-nu0-is-e-d}), where $y\in \overline{\Omega_n}$.
We now need to check that $v_n$ coincides with $v$ which was the solution to
(\ref{oper-when-nu0-is-e-d}) given by Theorem \ref{Thm-Prange}.
For that we will use the next lemma, which, as
well as the initial idea of exploiting layered structure of the problem were motivated by \cite{Neuss}.

\begin{lem}\label{Lem-sup-norms}
For $n\in \mathbb{S}^{d-1}$ satisfying $n_d >0$, let $v_n$ be
the solution to (\ref{oper-when-nu0-is-e-d}) constructed as above. Then 
$v_n \in C^\infty(\overline{\Omega_n}) \cap L^\infty(\overline{\Omega_n})$ and satisfies the following properties
\vspace{0.1cm}
\begin{itemize}
\item[{\normalfont{(a)}}]$ || \nabla v_n ||_{L^\infty  ( \{ y\cdot n >t) \} ) } \to 0, \qquad \text{as } t\to \infty$,
\vspace{0.2cm}
\item[{\normalfont{(b)}}]$ \int_0^\infty || (n\cdot \nabla) v_n ||^2_{L^\infty  ( \{ y\cdot n = t) \} ) } dt <\infty$.
\end{itemize}
\end{lem}

\begin{proof}
We have $v_n(y) = w_n(T_n^{-1} y)$ where  $y\in \overline{\Omega_n}$,
hence the up to the boundary smoothness of $v_n$
directly follows from that of $w_n$.

Let $G=(0,1)^{d-1}\times \R_+$ and fix some $\tau>0$ so that $w_n$ has finite  $V_\tau(G)$-norm.
Since $w_n$ solves (\ref{bdry-layer-system-halfspace2}),
where the coefficients and the boundary data have bounded $C^k$-norms for any $k\geq 0$,
by standard Schauder estimates near the boundary (see \cite[Theorem 5.21]{Giaq})
we have that $|\nabla_z w_n(z)|\leq C$ uniformly for all $z=(z', z_d) \in \R^{d-1}\times [0,\infty)$
satisfying $z_d\leq 1$. 
We now estimate $|\nabla_z w_n(z)|$ for $z=(z',z_d)\in \R^d_+$ with $z_d > 1$. By $\Z^{d-1}$-periodicity of $w_n(\cdot, z_d)$
we may assume that $z' \in (0,1)^{d-1}$. For $r>0$ let $K(z, r)$ be a closed cube
centred at $z$ and having side length $r$. In view of interior
Schauder estimates (see \cite[Theorem 5.19]{Giaq}) we have
\begin{equation}\label{Holder-est}
|| \nabla w_n ||_{L^\infty (K(z,1/4) )} \leq C  || \nabla w_n ||_{L^2(K(z,1/2))},
\end{equation}
with constant $C$ independent of $z$.
Set $K'(z,1/2)$ to
be the $(d-1)$-dimensional cube which is the projection of $K(z,1/2)$ onto $\R^{d-1}\times \{0\}$. We have
\begin{multline}\label{L2-norm-by-V-gamma}
|| \nabla w_n||_{L^2(K(z,1/2))}^2 = \int\limits_{ K'(z,1/2) } \int\limits_{ z_d - 1/2 }^{z_d+1/2} | \nabla w_n (x',x_d) |^2 dx' d x_d \leq \\
e^{\tau } e^{-2 \tau z_d } \int\limits_{ K'(z,1/2) } \int\limits_{ z_d - 1/2 }^{z_d+1/2}
|e^{\tau x_d} \nabla w_n (x', x_d) |^2 dx' d x_d \leq e^{\tau } e^{-2 \tau z_d } || w_n||_{V_\tau ( G )}^2,
\end{multline}
where we have used the periodicity of $w$ to get a bound in $V_\tau (G)$-norm.
From (\ref{L2-norm-by-V-gamma}) and (\ref{Holder-est}) we obtain
\begin{equation}\label{Holder-est-final}
|\nabla w_n(z) | \leq C e^{-\tau z_d} || w_n ||_{V_\tau(G)}, \qquad z\in \R^d_+,
\end{equation}
where we have also included the case of $z_d \leq 1$ in view of the uniform bound on the gradient.
Both assertions of the lemma follow directly from (\ref{Holder-est-final})
and the relation $\nabla_y v(y) = (T_n^{-1} )^t \nabla_z w_n(z)$, with $y=T_n z$
which is due to the change of variables formula.

Finally, by writing 
$$
w_n (z)=w_n(z', z_d) = v_0(z', 0) + \int_0^{z_d} \partial_t w_n (z', t) dt 
$$
and using (\ref{Holder-est-final}) we get $w\in L^\infty(\overline{\Omega_n})$,
and complete the proof of the lemma.
\end{proof}

By Lemma \ref{Lem-sup-norms}, $v_n$
gives a smooth and bounded solution to (\ref{oper-when-nu0-is-e-d}), and satisfies condition 1
of Theorem \ref{Thm-Prange}. But the solution with these properties is unique
according to Theorem \ref{Thm-Prange}.
Hence we have the following.

\begin{cor}\label{cor-solutions-coincide}
Fix $n\in \mathbb{S}^{d-1}$ such that $n_d > 0$, and assume that in (\ref{oper-when-nu0-is-e-d}) $A$ and $v_0$
are independent of $e_d$ and satisfy the usual
ellipticity, smoothness, and periodicity assumptions (A1)-(A3) of Section \ref{sec-Intro}.
Then the solution $v_n$ of (\ref{oper-when-nu0-is-e-d}) coincides
with the one given by Theorem \ref{Thm-Prange}.
\end{cor}

From the properties of $w_n$ we now deduce an expansion for $v_n$.
For $n \in \mathbb{S}^{d-1}$ satisfying $n_d>0$ let $w_n$ be the solution to (\ref{bdry-layer-system-halfspace2}) constructed as above. 
Then due to the periodicity condition we have
\begin{equation}\label{expansion-of-w-n}
w_n(z)  = w_n(z', z_d) = \sum_{\xi' \in \Z^{d-1}} c_{\xi'} (n; z_d) e^{2\pi i \xi' \cdot z'} = 
\sum_{\xi \in \Z^{d-1} \times \{ 0 \} } c_{\xi} (n; z_d) e^{2\pi i \xi \cdot z},
\end{equation}
where $(z', z_d) \in \R^{d-1} \times [0,\infty)$ and for $\xi=(\xi',0) \in \Z^{d-1} \times \{0\}$ we let
\begin{equation}\label{F-coeff-of-w-n}
c_\xi (n; z_d) = \int_{\T^{d-1}} w_n(z', z_d) e^{-2\pi i \xi \cdot z} d z' 
\end{equation}
be the $\xi$-th Fourier coefficient of $w_n (\cdot, z_d)$. 
By the construction of $w_n$ for any $\xi \in \Z^{d-1}\times\{0\}$ we have
\begin{equation}\label{c-xi-is-constant-on-boundary}
 c_\xi(n; 0 ) = c_\xi (v_0), \ \  \forall n\in \mathbb{S}^{d-1} \text{ satisfying } n_d>0,
\end{equation}
where $c_\xi(v_0)$ is the corresponding Fourier coefficient of the fixed boundary data $v_0$
involved in (\ref{oper-when-nu0-is-e-d}). 
The definition of $T_n$ yields
$$
z_d= T_n^{-1} y \cdot e_d = y \cdot (T_n^{-1})^t e_d = \frac{y \cdot n}{n_d},
$$
and $z' = y'$. Using these relations between $y$ and $z$, from (\ref{expansion-of-w-n})
for the solution of $v_n$ of (\ref{oper-when-nu0-is-e-d})
we obtain
\begin{equation}\label{exp-of-v-case-of-ed}
 v_n(y)  = w_n (T^{-1}_n y) = \sum_{\xi \in \Z^{d-1} \times \{ 0 \} } c_{\xi} \left(n; \frac{y \cdot n}{n_d} \right)
 e^{2\pi i \xi \cdot y}, \ \ \ y\in \overline{\Omega_n}.
\end{equation}

Observe that in view of the smoothness of $w_n$ the function $t \longmapsto c_\xi(n; t) \in \R^N$
is smooth on $[0,\infty)$ for each $\xi \in \Z^{d-1} \times \{0\}$. What we show
next is a stability result with respect to normal vector $n$ for the derivative of this function.

\begin{lem}\label{Lem-main}
Fix $\delta>0$ small, and let $\nu, \mu \in \Ss^{d-1}$ satisfy $\nu_d, \mu_d \geq \delta$.
Then there exists a constant $C_\delta =C(\delta, A ) $ such that
for any $t\geq 0$ and all $\xi \in \Z^{d-1}\times\{0\}$ one has
\vspace{0.1cm}
\begin{itemize}
 \item[{\normalfont(a)}]  $\sum_{\xi\in \Z^{d-1} \times \{ 0\} } | \partial_t c_\xi (\nu; t) | <\infty $
 \vspace{0.2cm}
 \item[{\normalfont(b)}] $  | \partial_t  c_\xi(\nu; t) - \partial_t  c_\xi(\mu ; t)    | \leq C_\delta | \nu - \mu| \times || v_0||_{ C^2 (\mathbb{T}^d ) } $ 
\end{itemize}
\vspace{0.1cm}
where $c_\xi(n; t)$ is given by (\ref{F-coeff-of-w-n}).
\end{lem}

\begin{proof}
We start with part (a). From \eqref{expansion-of-w-n} we have that $\partial_t c_\xi (\nu; t)$ is the $\xi$-th Fourier coefficient
of $\partial_d w_n (\cdot, t)$, which is a smooth and $\Z^{d-1}$-periodic function by construction.
Hence, we have (a). 

In order to establish stability estimate (b),
observe that thanks to (\ref{F-coeff-of-w-n}) it suffices to prove stability of $\nabla w_n$
with respect to $n$. Recall the notation $G= \mathbb{T}^{d-1} \times \R_+ $, and
set $u=w_\nu - w_\mu$.
We get that $u$ is a smooth solution to 
\begin{equation}\label{system-for-difference}
\begin{cases} -\nabla \cdot A_\nu(z) \nabla u (z) = \nabla \cdot F (z)  ,&\text{  $ z \in \R^d_+   $}, \\
 u  (z',0)= 0 ,&\text{  $ z' \in \R^{d-1}  $},
\end{cases}
\end{equation}
where we have denoted $ F (z):=( A_\nu (z) - A_\mu( z)) \nabla w_\mu(z)$.
Observe that
$u(\cdot, z_d)$, as well as $F(\cdot, z_d)$ are periodic with respect to $\Z^{d-1}$ for
any $z_d \geq 0$. Also, due to the construction it follows that $u\in V_\tau(G)$
for some $\tau>0$ which will be specified in a moment.
Since solution to (\ref{system-for-difference}) is unique in the space $V_\tau$,
we may apply estimate (\ref{app-thm-Lions-norm-est}) of Theorem \ref{thm-Lions-bdry-layers}
and by so obtain
\begin{equation}\label{u-diff-gamma-norm}
|| u ||_{V_\tau ( G; \hspace{0.05cm} \R^{d\times N} )} \leq 
\frac{C}{\tau} \frac{ || e^{\tau z_d} F (z)   ||_{L^2(G; \hspace{0.05cm} \R^{d\times N})} }{\lambda_{A_\nu} -2\tau || A_\nu ||_{L^\infty(\R^d) }  },
\end{equation}
where in (\ref{u-diff-gamma-norm}) we are following notation of Theorem \ref{thm-Lions-bdry-layers}.
Using (\ref{relation-of-A-B}) and (\ref{matrix-T-inverse}) from the definition of $A_\nu$ 
we have $|| A_\nu ||_{L^\infty(\R^d)} \leq C ||A ||_{L^\infty (\mathbb{T}^d) } \delta^{-2}$.
The latter combined with (\ref{ell-const-of-A-n}) implies that 
$\tau = \frac{ \lambda_A }{4 || A ||_\infty} \delta^4  $
is a valid choice in (\ref{u-diff-gamma-norm}),
where $\lambda_A$ is the ellipticity constant
of the original operator in (\ref{oper-when-nu0-is-e-d}).
Thus we will keep in mind that we have a uniform control over $\tau $ in terms of the threshold $\delta$.
Next, by (\ref{coeff-A-n}) and (\ref{matrix-T-inverse}) we easily get
\begin{equation}\label{diff-of-B-s}
|| A_\nu(  z ) - A_\mu( z) ||_{L^\infty(\R^d)} \leq C_\delta | \nu - \mu |,
\end{equation}
which in combination with the choice of $\tau$ and (\ref{u-diff-gamma-norm})
infers
\begin{equation}\label{u-est-1}
|| u ||_{V_\tau ( G )} \leq  C_\delta | \nu - \mu | \times  || w_\mu||_{V_\tau (G)} \leq C_\delta |\nu  - \mu | \times  || v_0 ||_{C^1( \mathbb{T}^d )},
\end{equation}
where the second inequality in (\ref{u-est-1}) is due to (\ref{app-thm-bdry-layer-bdry-data-est}).

Now fix some $z_0 \in \partial \R^d_+$, and for $r>0$ denote by $\mathcal{K}(z_0, r)$
the intersection of a cube with side length $r$ and center at $z_0$ with $\R^d_+$.
By boundary Schauder estimates (see \cite[Theorem 5.21]{Giaq} and its proof) we have
\begin{equation}\label{u-Holder-1}
|| \nabla u  ||_{ C^{0,\sigma} ( \mathcal{K}( z_0,1/2 )) } \lesssim_\delta 
 || \nabla u  ||_{ L^2 (\mathcal{K}(z_0, 1 )) }     + || F ||_{ C^{0,\sigma } (\mathcal{K}( z_0, 1  )) } ,
\end{equation}
where $0<\sigma <1$ is any fixed parameter, and the dependence of the constant in the inequality on parameter $\delta$ comes from the
dependence of the ellipticity constant of $A_{\nu} $ on $\delta$. It is clear that
\begin{equation}\label{u-L2-1}
|| \nabla u  ||_{ L^2 ( \mathcal{K}( z_0, 1 )) }  \leq   || u  ||_{V_\tau(G)}.
\end{equation}
Next, using the definition of $F $ we have
\begin{equation}\label{est-on-F}
|| F  ||_{ C^{0,\sigma } (\mathcal{K}( z_0, 1 )) } \leq || A_\nu ( z ) - A_\mu ( z ) ||_{ C^{0,\sigma}
  ( \mathcal{K}(z_0,1 )  )  } || \nabla w_\mu ||_{ C^{0,\sigma} ( \mathcal{K}(z_0,1 )  ) } .
\end{equation}
The first factor in the right-hand side of (\ref{est-on-F}) is easily seen,
as in (\ref{diff-of-B-s}), to be bounded by
$C_\delta |\nu-\mu|$. For the second one, we do a recourse to the construction of $w_\mu$
in Corollary \ref{cor-Lions} and again using Schauder estimates at the boundary we get
$$
 || \nabla w_\mu ||_{ C^{0,\sigma} ( \mathcal{K}(z_0,1 )  ) }  \lesssim_\delta 
|| \nabla w_\mu ||_{L^2( \mathcal{K}(z_0,2))  )} + || A_\mu \nabla v_0 ||_{ C^{0,\sigma} ( \mathcal{K}(z_0,2 )  )  } +
|| v_0 ||_{C^2 ( \mathcal{K}(z_0,2 ) )  }.
$$
In the last expression we estimate the $L^2$-norm of the gradient of $w_\mu$ by $V_\tau$ norm,
which, on its turn, is controlled by (\ref{app-thm-bdry-layer-bdry-data-est}).
Getting back to (\ref{est-on-F}) we obtain
\begin{equation}\label{F-est-final}
 || F  ||_{ C^{0,\sigma } (\mathcal{K}( z_0, 1 )) } \leq C_\delta || v_0 ||_{C^2(\mathbb{T}^{d} )} |\nu - \mu| .
\end{equation}
We now use (\ref{F-est-final}), (\ref{u-L2-1}) and (\ref{u-est-1}) in (\ref{u-Holder-1})
to get 
\begin{equation}\label{u-est-final}
|| \nabla u  ||_{ L^\infty ( \mathcal{K}( z_0,1/2 )) } \leq C_\delta || v_0 ||_{C^2(\mathbb{T}^{d} )} |\nu - \mu|.
\end{equation}

The claim (b) of the lemma now follows directly by taking the derivative under the integral sign in (\ref{F-coeff-of-w-n})
and applying (\ref{u-est-final}). The proof is complete.
\end{proof}

\subsection{Boundary layer correctors}\label{sec-bdry-layer-corr}
For irrational direction $n\in \mathbb{S}^{d-1}$ satisfying $n\cdot \nu_0>0$, and for fixed $1\leq \gamma\leq d$ let $v_n^{*,\gamma}$
be the solution to (\ref{bdry-layer-system-for-v-star})
in a sense of Theorem \ref{Thm-Prange}.
Under assumption (A5) on the operator
we apply $\nu_0 \cdot \nabla$
on both sides of the system in (\ref{cell-problem}) and get that $\chi^{*,\gamma}$,
the solution to the cell-problem,
is also independent of $\nu_0$. We next fix a $d\times d$ matrix $T_0$
with integer entries such that $T_0 e_d = \nu_0$
and\footnote{For our arguments it is enough to have existence of the inverse of $T_0$
with rational entries,
however, it is useful to see that 
with a little extra work one may assure $\mathrm{det} T_0 =1$
provided the greatest common divisor of the components of $\nu_0$ equals one (see Claim \ref{claim-det-1}).
The latter can always  be assumed without loss of generality, as the condition (A5)
is invariant under scaling of $\nu_0$. 
The advantage of having $\det T_0=1$ lies in the fact that the inverse of $T_0$ will also have integer entries,
which ensures that all boundary layer correctors $v^{*,\gamma}_n$
defined in (\ref{exp-for-bdry-layer-corr}) remain periodic with respect to $\Z^{d-1}$ in tangential directions.}
$\mathrm{\det} T_0 \neq 0$.
Making a change of variables in (\ref{bdry-layer-system-for-v-star}) by setting $y=T_0 z$,
and observing that $y \cdot n = z \cdot T_0^t n$,
we transform the problem for boundary layer corrector to
\begin{equation}\label{bdry-layer-corr2}
\begin{cases} -\nabla_z \cdot \widetilde{A}(T_0 z) \nabla_z \widetilde{v}_n^\gamma (z)  =0 ,&\text{  $z  \in \Omega_{T_0^t n} $}, \\
 \widetilde{v}_n^\gamma (z) =  -\widetilde{\chi}^\gamma (z)   ,&\text{  $ z \in \partial \Omega_{T_0^t n}  $,}    \end{cases}
\end{equation}
where we have set $v_n^{*,\gamma}(T_0 z)= \widetilde{v}_n^\gamma (z) $, $\chi^{*,\gamma}(T_0 z) = \widetilde{\chi}^\gamma (z)$,
and the coefficients are being transformed as in Section \ref{subsec-algebra}.
By the formula (\ref{deriv-x-y-matrix-T}) we have
$$
\nu_0 \cdot \nabla_y  = T_0 e_d \cdot (T_0^{-1})^t \nabla_z= e_d \cdot \nabla_z,
$$
hence both the operator and the boundary data in (\ref{bdry-layer-corr2}) are independent of the $d$-th coordinate.
Moreover, as $T_0$ has integer entries, it follows that coefficients of (\ref{bdry-layer-corr2})
as well as the boundary data are periodic with respect to $\Z^d$.
It is also clear that by the irrationality of $n$ and the choice of $T_0$ we have $T_0^t n \notin \R \Q^d$.
Also, $\widetilde{v}_n^\gamma(z)$ is the solution of (\ref{bdry-layer-corr2}) in a sense of Theorem \ref{Thm-Prange}
if and only if $v_n^{*,\gamma} (T_0 z)$ is the solution to (\ref{bdry-layer-corr2})
in a sense of Theorem \ref{Thm-Prange}.
Finally noticing that $T_0^t n \cdot e_d = n\cdot \nu_0 >0  $,
in (\ref{bdry-layer-corr2}) we are now in a position to apply the analysis of Section \ref{e-d-is-nu-0}.
In particular, from (\ref{exp-of-v-case-of-ed}) we get that
$\widetilde{v}_n^\gamma$, the solution to (\ref{bdry-layer-corr2}), has the following expansion
$$
\widetilde{v}_n^\gamma (z) =   \sum_{\xi \in \Z^{d-1} \times \{ 0 \} } c_{\xi}^\gamma \left( T_0^t n; \frac{z \cdot T_0^t n}{T_0^t n \cdot e_d} \right)
 e^{2\pi i \xi \cdot z}, \ \ \ z\in \overline{\Omega_{T_0^t n}},
$$
where Fourier coefficients $c_\xi^\gamma$ are defined in analogy with (\ref{F-coeff-of-w-n}),
in particular we have 
\begin{equation}\label{a0}
c_\xi^\gamma (T_0^t n; 0) =c_\xi (- \chi^{*,\gamma} ), \qquad \xi\in \Z^{d-1}\times \{0\}. 
\end{equation}
Since $T_0^{-1} y \cdot T_0^t n= y \cdot n$ and $T_0^ t \cdot e_d = n \cdot \nu_0$ we finally get
\begin{equation}\label{exp-for-bdry-layer-corr}
 v_n^{*,\gamma} (y) =   \sum_{\xi \in \Z^{d-1} \times \{ 0 \} } c_{\xi}^\gamma \left( T_0^t n; \frac{y \cdot n}{ n \cdot \nu_0} \right)
 e^{2\pi i (T_0^{-1})^t \xi \cdot y}, \ \ \ y\in \overline{\Omega_n},
\end{equation}
for the solution of (\ref{bdry-layer-system-for-v-star}).

Clearly, the entire analysis remains valid for irrational directions $n$ satisfying
$n \cdot \nu_0<0$.

\vspace{0.2cm}

\begin{proof}[Proof of Theorem \ref{thm-Main}]
For $\tau>0$ set $S_{\tau,+}=\{n\in \mathbb{S}^{d-1}: \hspace{0.1cm} n\notin \R \Q^d, \hspace{0.1cm} n\cdot \nu_0 >\tau \}$,
and for $\xi\in \Z^d$ consider the function $v_{\xi}(y): = e^{2\pi i \xi \cdot y} I_N$, $y\in \R^d$, where $I_N\in M_N(\R)$
is the identity matrix. 

Let $n\in S_{\tau, +}$, and consider a boundary layer system (\ref{bdry-layer-system}) 
set on $\Omega_n$ and with boundary data $v_{\xi}$. Let $v_{\xi}^\infty \in M_N(\R)$
be the corresponding constant field given by Theorem \ref{Thm-Prange}.
The formula (\ref{form-of-v-infty}) for $v_{\xi}^\infty $ in view of (\ref{avg-stability-1})
and (\ref{avg-stability-2}) is reduced to
\begin{multline}\label{v-infty-reduced-in-the-proof}
 v_{\xi}^\infty(n) = \int_{\partial \Omega_n } \partial_{y_\alpha} G^{0}(n,y) d\sigma(y) \times \big[
 c_{-\xi} (A^{\beta \alpha}) n_\beta +\\
 c_{-\xi} ( \partial_{y_\beta} (\chi^{*,\alpha})^t A^{\beta \gamma} ) n_\gamma +  
 \mathcal{M}\{  \partial_{y_\beta}(v_n^{*,\alpha})^t e^{2\pi i \xi \cdot y } A^{\beta \gamma}    \}  n_\gamma \big],
\end{multline}
where $c_\xi$ is the $\xi$-th Fourier coefficient, and $v_n^{*,\alpha}$ is the solution to
(\ref{bdry-layer-system-for-v-star}).
Indeed, $A^{\beta \alpha}$ is an $N\times N$ matrix for each $\alpha, \beta$, 
hence $A^{\beta \alpha} v_\xi = A^{\beta \alpha} e^{2\pi i \xi \cdot y}$. Consequently, the Fourier spectrum of $A^{\beta \alpha} v_\xi$ equals the Fourier spectrum of $A^{\alpha \beta}$
shifted by $- \xi\in \Z^d$, in particular we get $c_0(A^{\beta \alpha} (y) e^{2\pi i \xi \cdot y}) =c_{-\xi}(A^{\beta \alpha}) $.
The same argument applies to the second term in the brackets in (\ref{v-infty-reduced-in-the-proof}),
hence the reduction of (\ref{form-of-v-infty}) to (\ref{v-infty-reduced-in-the-proof}) follows.

To treat the term in (\ref{v-infty-reduced-in-the-proof}) involving boundary layer corrector 
we will apply Corollary \ref{cor-to-lem-quasi-per}.
Since $v_n^{*,\alpha}$ is smooth up to the boundary of $\Omega_n$ and has expansion (\ref{exp-for-bdry-layer-corr})
it follows that $\partial_{y_\beta}(v_n^{*,\alpha})$ has a similar expansion into exponentials
obtained from term by term differentiation of the series in (\ref{exp-for-bdry-layer-corr}).
Thus, if $\mathfrak{a}_{n}^{ \alpha \beta } (\eta; t)$, for $t\geq 0$ and $\eta\in \Z^d$, denotes the $\eta$-th coefficient of $\partial_{y_\beta}(v_n^{*,\alpha})$
for $t= y\cdot n$, we get
\begin{equation}\label{coeff-deriv}
 \mathfrak{a}_{n}^{ \alpha \beta } (\eta; \hspace{0.05cm} y\cdot n ) =  \frac{n_\beta}{n \cdot \nu_0} (\partial_t c_\eta^{\alpha}) \left(  T_0^t n; \frac{y\cdot n}{n\cdot \nu_0}  \right) 
 + 2\pi i  (T_0^{-1})^t \eta \cdot e_\beta \hspace{0.05cm} c_\eta^\alpha \left(  T_0^t n ; \frac{y\cdot n}{n\cdot \nu_0}  \right) ,
\end{equation}
for all $y\in \overline{\Omega_n}$, where $e_\beta$ is the $\beta$-th vector in the standard basis of $\R^d$.
But recall, that $\mathcal{M}$-averages are understood for restrictions of functions on the boundary of
halfspace $\Omega_n$ (see the discussion after \eqref{form-of-v-infty}). Hence, in order to be able to apply Corollary \ref{cor-to-lem-quasi-per}
for the set of parameters $S_{\tau,+}$ we need only to consider dependence of (\ref{coeff-deriv})
on $n$ for $y\cdot n =0$. 

Observe that from (\ref{exp-for-bdry-layer-corr}) we have $\mathfrak{a}_{n}^{\alpha \beta} (\eta; \hspace{0.05cm} t)\equiv 0$ for any $\eta \in \Z^d$
with $\eta_d \neq 0$.
Thanks to (\ref{a0}) we get
\begin{equation}\label{coeff-deriv0}
 \mathfrak{a}_{n}^{ \alpha \beta } (\eta; 0 ) =  \frac{n_\beta}{n \cdot \nu_0} (\partial_t c_\eta^\alpha) (  T_0^t n; 0) 
 + 2\pi i (T_0^{-1})^t \eta \cdot e_\beta \hspace{0.05cm} c_\eta (- \chi^{*,\alpha} )  ,
\end{equation}
thus we need to check the following two conditions in order to apply Corollary \ref{cor-to-lem-quasi-per},
namely for any $1\leq \alpha, \beta \leq d$ we must have
\begin{align}
   \label{c1} \sum_{\eta \in \Z^{d-1} \times \{0\} } | \mathfrak{a}_{n}^{ \alpha \beta } (\eta; 0 ) | &< \infty \text{ for all } n\in S_{\tau, +},   \\
   \label{c2} \sup_{\eta \in \Z^{d-1} \times \{0\}}|\mathfrak{a}_{n^{(1)}}^{ \alpha \beta } (\eta; 0 ) - \mathfrak{a}_{ n^{(2)} }^{ \alpha \beta } (\eta; 0 ) | &\leq C_0 |n^{(1)} - n^{(2)}|   
   \text{ for any } n^{(1)}, n^{(2)} \in S_{\tau, +} \ .
\end{align}
Due to the choice of $T_0$ we have $T^t_0 n \cdot e_d = n \cdot T_0 e_d = n\cdot \nu_0 > \tau$ for all $n\in S_{\tau, +}$.
Using this we apply Lemma \ref{Lem-main} part (a) and from (\ref{coeff-deriv0}) and the smoothness of $\chi^{*,\alpha}$ - the solution to cell-problem,  we obtain \eqref{c1}.
Next, by Lemma \ref{Lem-main} part (b) and (\ref{coeff-deriv0}) we arrive at \eqref{c2}.

Applying Corollary \ref{cor-to-lem-quasi-per} with the smooth function $e^{2\pi i \xi \cdot y} A^{\beta \gamma}$
and for each parameter $n\in S_{\tau, +}$ choosing the set of coefficients $\{\mathfrak{a}_{n}^{ \alpha \beta } (\eta; 0 ) \}_{\eta \in \Z^{d-1}\times \{0\}}$
(as the coefficients of the expansion for functions in Corollary \ref{cor-to-lem-quasi-per}),
we obtain that the mapping
$n\longmapsto \mathcal{M}\{  \partial_{y_\beta}(v_n^{*,\alpha})^t e^{2\pi i \xi \cdot y } A^{\beta \gamma}    \}$
is Lipschitz continuous on $S_{\tau, +}$ with Lipschitz constant bounded by a constant $C_\tau=C (\tau, A, d)$,
independent of $n$ and $\xi$. Finally, combining this with Lemma \ref{Lem-Green-reg-systems},
from (\ref{v-infty-reduced-in-the-proof}) we get 
\begin{equation}\label{const-field-is-Lip}
|v_{\xi}^\infty(n) - v_{\xi}^\infty(\nu ) | \leq C_\tau |n-\nu|, \qquad n, \nu \in S_{\tau, +},
\end{equation}
where $C_\tau$ is independent of $\xi$.

For $\tau>0$ define $D_{\tau, +} = \{x\in \partial D: \hspace{0.1cm} n(x) \notin \R \Q^d, \hspace{0.1cm} n(x) \cdot \nu_0>\tau \}$,
where $n(x)$ is the normal inward vector of $\partial D$ at $x$.
Following the notation of Section
\ref{sec-contr-of-g}, for any $x,y\in \partial D_{\tau, +} $ by (\ref{g-as-a-series}) we have
\begin{multline}\label{g*-diff}
| g^*(x) - g^*(y) |  \leq \sum_{\xi \in \Z^d} | g_\xi^*(x)  - g_\xi^* (y) | \leq
 \sum_{ \xi \in \Z^d} |c_\xi (x) | \times | v_{\xi}^\infty(n(x))  - v_{\xi}^\infty(n(y))  | +  \\
\sum_{\xi \in \Z^d} | v_{\xi}^\infty(n(x)) | \times | c_\xi(x) - c_\xi(y) |  =: 
\Sigma_1 + \Sigma_2.
\end{multline}
Recall that $c_\xi (x) = \int_{\mathbb{T}^d} g(x, z) e^{-2\pi i \xi \cdot z} dz $, where $\xi\in \Z^d$, $x\in \partial D$.
Fix a non-zero $ \xi=(\xi_1,...,\xi_d) \in \Z^d$ and let $|\xi_\alpha| = \max_{1\leq \beta \leq d} |\xi_\beta|$.
Let also $\partial_{2,\alpha}^{d+1}$ be the partial differentiation operator acting
on $g(x, \cdot)$ $(d+1)$-times in the $\alpha$-th coordinate. Using the smoothness of $g$,
from the definition of $c_\xi(x)$ we get
$$
 c_\xi(x) = \frac{1}{(-2\pi i \xi_\alpha)^{d+1}} \int_{\mathbb{T}^d} \partial_{2,\alpha}^{d+1} g(x, z) e^{-2\pi i \xi \cdot z} dz.
$$
Combining this with a uniform bound on $|c_0(\cdot)|$ we get
\begin{equation}\label{coeff-est-1}
 |c_\xi(x)|\lesssim_g (1+|\xi|)^{-(d+1)} \text{ uniformly in } x\in \partial D \text{ and } \xi \in \Z^d.
\end{equation}
As $g$ is smooth with respect to both of its variables, in a similar way we obtain
\begin{equation}\label{coeff-est-2}
 |c_\xi(x)-c_\xi(y)|\lesssim_g (1+|\xi|)^{-(d+1)} |x-y| ,
\end{equation}
for all $x,y\in \partial D$ and non-zero $\xi\in \Z^d$. Using (\ref{coeff-est-1}) and (\ref{const-field-is-Lip})
for any $x,y \in \partial D_{\tau, +}$ we get 
$\Sigma_1 \leq C_{g, \tau} |n(x) - n(y)| \leq C_{g, \tau} |x-y|$,
where we have used the smoothness of $\partial D$ to obtain the second inequality.
In a similar vein, in $\Sigma_2$ using Lemma \ref{lem-bddness of boundary tail}
to bound the constant field and employing (\ref{coeff-est-2}) leads to
$\Sigma_2 \leq C_g |x-y|$.
The estimates for $\Sigma_1$ and $\Sigma_2$ applied to (\ref{g*-diff}) show that $g^*$
is Lipschitz on $\partial D_{\tau, +}$.

Obviously, the same argument works for the other
hemisphere $S_{\tau, -}=\{n\in \mathbb{S}^{d-1}: \ n\notin \R \Q^d, \  n\cdot \nu_0 < - \tau \}$
as well. The proof of the Theorem is now completed.
\end{proof}

\subsection{Concluding remarks}\label{sec-concluding}
Results concerning regularity of boundary layer tails are very few in the literature.
In the same setting as we have here, namely second order divergence type elliptic systems,
the smoothness of $g^*$ under restrictive condition (\ref{div-free-condition}) on the coefficients
and for $d\geq 3$, was established by H. Shahgholian, P. Sj\"{o}lin,
and the current author in \cite{ASS2} as an outcome of methods of \cite{ASS2} and \cite{KLS}
(see  formula (4.4) in \cite{ASS2}, and the discussion after that).
It is easy to see that we recover this result for $g^*$ from the proof of Theorem \ref{thm-Main} above.
Namely, the condition (\ref{div-free-condition}) implies that solutions to cell-problem (\ref{cell-problem}),
and hence to boundary layer systems (\ref{bdry-layer-system-for-v-star}), are trivial.
This in its turn shows that in formula (\ref{v-infty-reduced-in-the-proof})
the last average is vanishing, and we get that the boundary layer tail,
as a function of normal $n$, equals to a $C^\infty$ function almost everywhere on the sphere.
The rest of the proof proceeds with minor modifications. 
In dimension two, the smoothness of $g^*$ is new, while for $d\geq 3$
we get an alternative proof of the mentioned result from \cite{ASS2}.

Concerning other settings, the reader may consult a recent work by Feldman and Kim \cite{Feldman-Kim},
and the references therein,
where they analyse continuity properties of boundary layer tails associated with fully nonlinear
uniformly elliptic equations of second order.

Getting back to our case, one can see from the analysis above that the main obstacle
towards the regularity of $g^*$ comes from boundary layer correctors,
in particular we do not know if the behaviour of boundary layer tails $v^\infty(n)$
is in any sense uniform with respect to normals $n$.
A specific instance of this non-uniformity is the convergence speed of boundary layer correctors
to their corresponding tails away from the boundary.
Concerning this aspect in \cite{A-slow}
we show that given any one-to-one, continuous function decreasing to 0 at infinity
(i.e. a convergence rate), one may construct a problem of form (\ref{bdry-layer-system})
with smooth data, so that convergence towards boundary layer tail is slower than the given rate in advance.
This in particular indicates that approaches toward regularity of $g^*$
based on controlling the speed of convergence of the tails, are unlikely to lead to a positive conclusion.

It is also interesting to observe (in the
light of Example \ref{example-Laplace}) that condition (A5)
implies that the operator only ``sees'' Diophantine directions 
on the hemispheres considered in the proof of Theorem \ref{thm-Main}.
It thus leads to an idea that one may try to tailor the Diophantine condition of \cite{GM1}
to the given operator. Developing this line it seems plausible that 
one should be able to 
deduce the claim of Theorem \ref{Thm-our2} (although without any structural results
such as expansion (\ref{exp-for-bdry-layer-corr})) using instead methods of \cite{GM1}
combined with some of the ideas considered here, in particular Lemma \ref{lem-smooth-rotation} and the proof
of Theorem \ref{thm-Main}. 
In this perspective the approach of Section \ref{Sec-proof} should be seen as a more transparent alternative
to some of the methods of \cite{GM1} under condition (A5), and it will be interesting to see
if the ideas considered here can be developed to lead to an actual homogenization
of the problem (\ref{problem-osc-oper})-(\ref{prob-Dir-data-osc}) under conditions (A1)-(A5).

\appendix

\section{On the choice of transformation matrices}

This appendix contains two results concerning the choice of transformation matrices used
in Seciton \ref{sec-Green} and subsection \ref{sec-bdry-layer-corr},
which can be useful in further refining and extending the analysis of the present paper.

\subsection{Smooth rotations}\label{sub-rot-Lip}
Here we analyse the choice of orthogonal matrices $M$ sending $e_d$ to $n\in \mathbb{S}^{d-1}$ considered in Section \ref{sec-Green}.
The main purpose is to show, in a constructive fashion, that in a neighbourhood of a given $n\in \mathbb{S}^{d-1}$ there is a possibility of selection of these matrices
varying smoothly with $n$. Interestingly such a smooth selection
globally on $\mathbb{S}^{d-1}$
in general dimensions is not available due to topological obstructions discussed below.
Availability of such a choice can be used, for example, in the analysis of regularity with respect to normals $n$ of Green's matrices $G^{0,n}$
studied in Section \ref{sec-Green} (see Claim \ref{claim-Green-relat} for the change in the coefficient field introduced by $M$).

Recall that for each $n\in \mathbb{S}^{d-1}$ we choose a matrix $M \in \OO(d)$ such that $M e_d = n$.
Such $M$ is clearly of the form $M=[N|n]$, where $N$ is
$d\times (d-1)$ matrix  with the property that its columns form an orthonormal basis in the tangent space of $\mathbb{S}^{d-1}$
at the point $n$, in particular $M$ is defined modulo group $\OO(d-1)$. 
From this we see that the existence of orthogonal matrices $M$
sending $e_d$ to $n$ and varying smoothly with $n$
is equivalent to existence of a family of smooth vector fields $\{ \textbf{v}_1(n),...,\textbf{v}_{d-1}(n) \}_{n\in \mathbb{S}^{d-1}}$
that will form an orthonormal basis in the tangent space of $\mathbb{S}^{d-1}$ at any point $n$.
The existence of the desired vector fields, however,
is false in general\footnote{For $\R^d$, with $d\geq 3$ odd, the non-existence directly follows from Hairy Ball Theorem,
which states that there is no non-vanishing continuous, let alone smooth, tangent vector field on even-dimensional spheres.}.
Let us very briefly give some details and background on this matter. 

A $C^\infty $-manifold $X$ of dimension $d\geq 1$ is called \emph{parallelizable} if there exist
smooth vector fields $\{\textbf{v}_1(x),...,\textbf{v}_d(x) \}_{x\in X}$, such that at each point $x\in X$
the $d$-tuple $\{\textbf{v}_i(x)\}_{i=1}^d$ forms a basis in the tangent space of $X$
at $x$. It is well-known that a manifold is parallelizable if and only if its tangent bundle is trivial.
On the other hand the tangent bundle of the sphere $\mathbb{S}^{d-1}$ is trivial if and only if
$d=1,2,4,8$. We refer an interested reader to works by Bott, Kervaire, and Milnor \cite{Bott-Milnor}, \cite{Kervaire-Milnor}
for details and proofs.
 Notice, that parallelizabilty does not require the basis to be orthonormal,
nonetheless, it follows directly that when 
$d\notin \{1,2,4,8 \}$ one cannot fix a family of orthogonal matrices,
such that $M e_d= n$, and $M$ varies smoothly with respect to $n$ \emph{globally} on $\mathbb{S}^{d-1}$.
However, the existence of these smooth fields locally,
in the neighbourhood of each point $n\in \mathbb{S}^{d-1}$ is true, for which 
we give an elementary, self-contained constructive proof in the next lemma.

\begin{lem}\label{lem-smooth-rotation}{\normalfont{(Smooth selection of rotations)}}
Fix any point $p=(p_1,...,p_d) \in \mathbb{S}^{d-1}$. Then, there exists an open neighbourhood $\mathcal{P}\subset \mathbb{S}^{d-1}$
of $p$, and an assignment $n \longmapsto M_n $ from $\mathcal{P}$ into $\mathrm{O}(d)$ 
such that for all $n \in \mathcal{P}$ we have $M_n e_d = n$, and
for each $1\leq i, j\leq d$, the real-valued function $(M_n)_{ij}$ is $C^\infty$ on $\mathcal{P}$.
\end{lem}
\begin{proof}
The proof is by induction on dimension $d$. 
Assume that $p_1 \neq 0$, and fix a neighbourhood $\mathcal{P}$ of $p$
on $\mathbb{S}^{d-1}$ where $|n_1|> |p_1|/2$, for all $n\in \mathcal{P}$. Otherwise, if $p_1=0$
one may simply permute the coordinate system so that after the permutation the first coordinate of $p$ is non-zero. 
Thus there is no loss of generality in assuming that $p_1 \neq 0$.

Let $d\geq 2$ be fixed. To a given $n=(n_1,...,n_d) \in  \mathcal{P}$ we wish to assign a $d\times d$ matrix
$X_d(n)$ of the form
\begin{equation}\label{X-matrix}
 X_d(n) = \left(
\begin{array}{lllll}
a_1(n) & \ast & ... &   \ast & n_1\\
 & \cdots  &   & 1 & n_2\\
\vdots &   \ast & \text{\reflectbox{$\ddots$}} &  & \vdots \\
a_{d-1}(n) & 1 &   & \text{\Huge0} & n_{d-1}\\
1 & 0 & \cdots & 0 & n_d
          \end{array}
  \right)
\end{equation}  
where the opposite main diagonal is identically one except the element on the first row;
also, with the exception of the $(d,d)$-th element everything below the opposite main diagonal is identically zero,
and the rest of the elements above that diagonal are chosen so that
to have the following properties:  

\begin{doublespace}
\begin{itemize}
 \item[(1)] $X_d (n) e_d = n$,
 \item[(2)] all columns of $X_d(n)$ are pairwise orthogonal to each other,
 \item[(3)] for any $1\leq i,j \leq d$, the real-valued function $n\longmapsto [X_d(n)]_{ij}$ is smooth on $\mathcal{P}$,
where $[X_d(n)]_{ij}$ is the $(i,j)$-th element of the matrix $X_d(n)$.
\end{itemize}
\end{doublespace}

Start with $d=2$, and let $n=(n_1,n_2) \in \mathcal{P}$ be any. Consider the matrix
$X_2 (n) =\left(
      \begin{array}{cc}
        a_1(n) & n_1 \\
        1 & n_2 \\
      \end{array}
    \right)
$, where we have chosen $a_1(n)=-n_2/n_1$. Obviously $X_2 (n)$ is of the form (\ref{X-matrix}), and satisfies properties (1)-(3) listed above.
Now assume we have this construction for dimension $d-1$, and let us construct for $d$.  
We take $n=(n_1,...,n_d) \in \mathcal{P}$, and set
$$
 X_d(n) =  \left(
      \begin{array}{ccccc}
        a_1(n) &   &  & & \\
         \vdots & & & & \\
	 & & \text{\Large $X_{d-1}(n_1,...,n_{d-1})$} & & \\
	 a_{d-1}(n) & & & & \\
	 1 & 0 & \cdots  & 0 & n_d \\
      \end{array}
    \right),
$$
where the vector field $A_d(n):= (a_1(n),...,a_{d-1}(n)) $ will be chosen in a moment.
By our construction and inductive hypothesis we have that $X_d(n)$ is of the form (\ref{X-matrix}), whatever the choice of the field $A_d$ is,
and hence in particular, condition (1) above is automatically satisfied.
Again in view of the inductive hypothesis and the construction of $X_d(n)$,
starting from the second one all $(d-1)$ columns of $X_d(n)$
satisfy (2) and (3). It is left to determine the field $A_d(n)$. Observe that 
the first column of $X_d(n)$ is orthogonal to the rest of $(d-1)$ columns if and only if $A_d(n)$ satisfies
\begin{equation}\label{system-for-A}
A_d(n) X_{d-1}(n_1,...,n_{d-1}) = (0,...,0,-n_d) \in \R^{d-1},
\end{equation}
where we have treated $A_d(n)$ as a row-vector.
On one hand for each fixed $n\in \mathcal{P}$, (\ref{system-for-A}) is a system of linear equations
with respect to unknowns $(a_1,...,a_{d-1})$,
and with matrix of coefficients equal to $X_{d-1}$. On the other hand, by inductive hypothesis we have that all columns of $X_{d-1}$
are pairwise orthogonal, moreover, by (\ref{X-matrix}) we see that all columns of $X_{d-1}$ considered as $(d-1)$-dimensional vectors
have lengths uniformly bounded away from zero when $n \in \mathcal{P}$.  
This, in particular, shows that the determinant of $X_{d-1}$,
which in this case will be the product of the lengths of its column-vectors in view of the orthogonality
condition, will stay away from zero uniformly as $n\in \mathcal{P}$. 
We thus conclude that the system (\ref{system-for-A}) is uniquely solvable for all $n \in \mathcal{P}$,
and solutions are smooth functions in $n$ due to inductive hypothesis applied to $X_{d-1}$,
and Cramer's rule concerning systems of equations. 
All properties (1)-(3) are now fulfilled, and inductive step is completed.

It is now left to normalize each column of $X_d(n)$ to unit length, so that to get an orthogonal matrix.
For each $n \in \mathcal{P}$ we let $M_n$ be the matrix obtained from $X_d(n)$ where we divide all elements on the given column
of $X_d$ by the Euclidean length of that column-vector.
It is important to observe, that the last column of $X_d(n)$, which is the vector $n$, is of unit length, thus it will remain unchanged
leaving the condition of sending $e_d$ to $n$ unaltered. The rest of 
all other $d-1$ columns of $X_d(n)$ have length at least one, hence this normalization will not affect the smoothness of 
the individual components of the matrix.
It now follows that the mapping $\mathcal{P} \ni n \longmapsto M_n \in \OO(d)$ satisfies all requirements of the lemma.
The proof is complete.
\end{proof}

\subsection{An element of $\mathrm{SL}(d, \Z)$ with prescribed column} 
We show that the integer matrix $T_0$ which was used to transform $e_d$ to the given vector $\nu_0$
in subsection \ref{sec-bdry-layer-corr}
can be chosen satisfying $\det T_0 =1$, implying that its inverse also has integer elements.
This fact can be used to get periodicity in tangential directions of boundary layer correctors considered in (\ref{exp-for-bdry-layer-corr}).
For given integers $(a_1,...,a_d)\in \Z^d$ we denote by $[a_1,...,a_d] $ their greatest common divisor.
We also recall a standard notation $\mathrm{SL}(d,\Z)$ for the \emph{special linear group} over integers.
\begin{claim}\label{claim-det-1}
For any non-zero $a=(a_1,...,a_d)\in \Z^d$ such that $[a_1,...,a_d] =1$ there exists
$T\in \mathrm{SL}(d, \Z)$ satisfying $T e_d = a$.
\end{claim}

\begin{proof}
Before we start, observe that the condition on greatest common divisor to be 1
is necessary which trivially follows from Euclid's algorithm.

It is enough to consider the case of $a\in \Z^d$ having at least one of its last two coordinates non-zero.
Indeed, assume the claim holds for that class of $d$-tuples, and take any $a\in \Z^d$ with $a_{d-1}=a_{d}=0$ and satisfying the condition of the claim.
Then, we necessarily have $d\geq 3$, and hence can fix $1\leq i \leq d-2$ such that
$a_i \neq 0$. Clearly, one of the transpositions $(i,d-1)$ or $(i,d)$ is even.
Assume the second one, and consider $\widetilde{a} \in \Z^d$ which is obtained from $a$ by swapping $i$-th coordinate with $d$-th,
and keeping the rest unchanged. Now, if $\widetilde{T}$ is the matrix for $\widetilde{a}$ satisfying the claim, then
the matrix $T$ obtained from $\widetilde{T}$ by swapping its $i$-th row with $d$-th  satisfies $\mathrm{det} T = \mathrm{det} \widetilde{T}=1$, since the transposition $(i,d)$ was even.
Thus $T$ satisfies the claim for the original $a\in \Z^d$. Given this, we will only consider $a\in \Z^d$
satisfying $|a_{d-1}| + |a_d|>0$.

Next, we claim that for each $d\geq 2$ and any $a=(a_1,...,a_d) \in \Z^d$ with $|a_{d-1}|+|a_d|>0$,
and $[a_1,...,a_d]=1$
there exists a matrix $T\in \mathrm{SL}(d, \Z) $ having $a$ as its last column and such that all elements of $T$
above the main diagonal, except possibly on the last column, are 0.
The proof of this statement proceeds by induction on $d$. 

The case of $d=2$ follows directly from Euclid's algorithm.
Now assume the induction hypothesis holds for $d\geq 2$, and take $a\in \Z^{d+1}$
such that at least one of its last two coordinates is non-zero and $d+1$ coordinates of $a$ are coprime.
Obviously, the inductive hypothesis applies to $(a_1,..., a_{d-1}, [a_d, a_{d+1}]) \in \Z^d$
and we let $T^{(0)}$ be the corresponding $d \times d$ matrix. 
By Euclid's algorithm there are $x,y\in \Z$ such that $[a_d, a_{d+1} ] = x a_d + y a_{d+1}$.
Clearly $[x,y]=1$. Now, consider a $(d+1)\times (d+1)$ matrix $T$ of the form
\begin{equation}
 T =  \left(
      \begin{array}{cccccc}
	 & & &  & 0 & a_1 \\
	 & & \text{\Large $T^{(0)}_{d-1,d-1}$} & & \vdots & \vdots \\
	 & & &  & 0  & a_{d-1} \\
	 t_{d1}     & t_{d 2}     & \ldots & t_{d \hspace{0.05cm} d-1}   & y  & a_d \\
	 t_{d+1 \hspace{0.05cm} 1} & t_{d+1 \hspace{0.05cm} 2} & \ldots & t_{d+1 \hspace{0.05cm} d-1} &  -x & a_{d+1} \\
      \end{array}
    \right),
\end{equation} 
where $T^{(0)}_{d-1,d-1}$ is the submatrix of $T^{(0)}$ on the first $d-1$ rows and columns, and $t_{ij}\in \Z$, $i=d,d+1$ and $1\leq j \leq d-1$,
are parameters to be chosen below. To complete the induction step, and hence the proof of the entire claim,
it remains to show that there is a choice of $t_{ij}$ ensuring $\det T =1$.
Expanding the determinant of $T$ with respect to its last two rows (using standard extension of the Laplace expansion) we get
\begin{equation}\label{T-matrix-exp}
\mathrm{det} T = \sum_{1\leq i< j\leq d+1} \e_{ij} \mathrm{det} M_{ij}  \mathrm{det} \widetilde{M}_{ij},
\end{equation}
where $\e_{ij} =(-1)^{d+(d+1)+i+j} $,
$M_{ij}$ is the $2 \times 2$ submatrix of $T$ from elements on the rows $d, d+1$ and columns $i,j$,
and $\widetilde{M}_{ij}$ is the $(d-1)\times (d-1)$ submatrix complementary to $M_{ij}$.
Due to construction, $\widetilde{M}_{ij}$ contains a zero-column unless $i$ or $j$ in 
(\ref{T-matrix-exp}) equals $d$. Hence,
(\ref{T-matrix-exp}) reduces to 
\begin{equation}\label{T-matrix-exp-revised}
\det T = \sum_{1\leq i \leq d-1} \e_{i d} \det \left( \begin{array}{ll}
t_{d i} & y \\
t_{d+1 \hspace{0.05cm} i} & -x
   \end{array}
\right)  \det \widetilde{M}_{i \hspace{0.05cm} d}  + [a_d, a_{d+1}] \det \widetilde{M}_{d \hspace{0.05cm} d+1}
\end{equation}
For each $1\leq i \leq d-1$ we choose $t_{di}, \ t_{d+1 \hspace{0.05cm} i} \in \Z$ in order to have
$ \e_{id } ( -t_{d i } x - y t_{d+1 \hspace{0.05cm} i}  ) =(-1)^{i+d}  T^{(0)}_{di}  $,
where $T^{(0)}_{di}$ is the $(d,i)$-th element of $T^{(0)}$. The existence of this choice follows from equality $\Z=  x \Z + y \Z$,
which on its turn follows directly by Euclid's algorithm and the fact that $x,y$ are coprime.
Now with this choice of parameters $t_{ij}$, (\ref{T-matrix-exp-revised}) coincides with the expansion
of $\det T^{(0)}$ with respect to its last row, and hence equals 1. The proof of the claim is complete.
\end{proof}

\vspace{0.2cm}

\noindent \footnotesize{\textbf{Acknowledgements.}
This note is partially based on my PhD thesis completed at
The University of Edinburgh in 2015, and I would like to thank my thesis advisor
Dr. Aram Karakhanyan for a number of stimulating and useful discussions.
I also thank Prof. Henrik Shahgholian for valuable comments regarding the manuscript.
Part of the work was conducted while I was visiting Institut Mittag-Leffler during the term
``Homogenization and Random Phenomenon". I thank The Institute for its warm hospitality and support.
I am grateful to the anonymous referee whose detailed comments and remarks  helped me to improve the exposition
of the paper significantly.
This article is finalized at KTH, where I am supported by postdoctoral fellowship from Knut and Alice Wallenberg Foundation.}

\end{document}